\documentclass{ws-ijac}

\PassOptionsToPackage{a4paper}{geometry}

\usepackage{url,cite}
\usepackage{xcolor}
\usepackage[verbose]{hyperref}
\hypersetup{colorlinks=false,allbordercolors=blue,pdfborderstyle={/S/U/W 1}}

\usepackage{multicol}
\usepackage{amssymb}
\usepackage{tikz,tkz-tab}              
\usetikzlibrary{cd}                
\usetikzlibrary{arrows.meta, decorations.pathmorphing}

\usepackage{amsmath}
\usepackage{verbatim}
\usepackage{csquotes}

\numberwithin{equation}{section}

\newcommand{\N}{\mathbb{N}}

\newcommand{\wl}[1]{\lvert#1\rvert}
\def\VR(#1,#2){\vrule width0pt height#1mm depth#2mm}
\newcommand{\ldiv}{\le}
\newcommand{\ldivt}{\mathrel{\widetilde{\le}}}
\newcommand{\ldivset}{D}
\newcommand{\rdivset}{\widetilde{D}}
\newcommand{\strldiv}{<}
\newcommand{\strldivt}{\mathrel{\widetilde{<}}}
\newcommand{\factor}{\subseteq}
\newcommand{\strfactor}{\mathrel{\subset}}
\renewcommand{\gcd}{\wedge}
\newcommand{\gcdt}{\mathrel{\widetilde\gcd}}
\newcommand{\lcm}{\vee}
\newcommand{\lcmt}{\mathrel{\widetilde\lcm}}

\newcommand{\pres}[2]{\left\langle #1 \middle\vert #2 \right\rangle}

\newcommand{\eqeq}{\simeq}
\newcommand{\weq}{\equiv}
\newcommand{\meq}{\def\arraystretch{0}\begin{array}{c}\overline{\overline{\,.\,}}\end{array}}

\newcommand{\defit}[1]{\emph{#1}}
\newcommand{\rcb}{\theta}   
\newcommand{\rc}[2]{\rcb(#1,#2)} 
\newcommand{\ercb}{\rcb^*} 
\newcommand{\erc}[2]{\ercb(#1,#2)} 
\newcommand{\erctb}{\ercb_3}   
\newcommand{\erct}[3]{\erctb(#1,#2,#3)} 
\newcommand{\fin}[1]{\overline{#1}} 
\newcommand{\rr}{\curvearrowright} 
\newcommand{\rf}{\rr^*\supset} 
\newcommand{\strrf}{\rr^+\supset} 

\begin{document}

\begin{center}
\large \textbf{Computing least common multiples in monoids with a finite Garside family}\\
\vspace{0.05\textheight}
\normalsize Emir Melliti\\
\small\textit{Institut de Mathématiques de Jussieu - Paris Rive Gauche}\\
\textit{École Normale Supérieure - PSL}\\
\textit{emir.melliti@ens.psl.eu}
\vspace{0.05\textheight}
\small
\begin{center}
\begin{minipage}{0.85\textwidth}
\small

\textbf{Abstract.} Right-reversing is an algorithm used to compute least common multiples in monoids that admit a right-complemented presentation. The algorithm can either terminate and find a result, fail, or run indefinitely. The correctness of the algorithm can be proved with additional assumptions coming from Garside theory. In the same framework, we prove that a non-terminating run of the algorithm is necessarily cyclic. Stopping when a cycle is detected provides a way of computing a minimal Garside family.
\end{minipage}

\end{center}
  
\end{center}
\.

\section{Introduction}
    \subsection{Context}
    A right-complemented presentation of a monoid is a presentation $\pres{S}{R}^+$ where the relations in $R$ have the form $s\rcb(s,t) = t\rcb(t,s)$, where $\rcb$ is a partial map from $S\times S$ to the free monoid $S^*$ . Introduced in \cite{Dehornoy1994BraidGA}, the main example for which this notion was defined was to study braid groups and monoids, and more generally Artin-Tits groups and monoids. We recall the definition of Artin-Tits monoids:
    \begin{definition}
        A \defit{Coxeter matrix} $\mathcal{M}$ associated with a finite set of generators $S$ is a symmetric square matrix whose entries, $m_{a,b}$ for all $a,b\in S$, are integers satisfying $m_{a,b} \ge 2$ or $m_{a,b} =\infty$ for all $a \ne b \in S$ and $m_{a,a}=1$ for $a\in S$. We often represent a Coxeter matrix $M$ with its associated \defit{Coxeter graph} $\Gamma=\Gamma_{\mathcal{M}}$, an undirected graph whose vertices are the elements of $S$ and two generators $a,b$ are linked by an edge if $m_{a,b}\ge 3$ and the edge is labelled by $m_{a,b}$ if $m_{a,b} \ge 4$.
    \end{definition}\par
    With a Coxeter matrix $\mathcal{M}$, we will associate a set of relations that will define the \defit{Artin-Tits} monoid of type $\mathcal{M}$.
    \begin{definition}
        The Artin-Tits monoid $M=M_{\Gamma}=M_{\mathcal{M}}$ associated with a Coxeter matrix $\mathcal{M}$ (and a Coxeter graph $\Gamma$) is the monoid with the presentation $\pres{S}{R}^+$, where $R$ is the set of relations of the form:
        \[
        \underbrace{aba\ldots}_{m_{a,b} \text{ factors}} = \underbrace{bab\ldots}_{m_{a,b} \text{ factors}}
        \]
        for all $a, b \in S$ such that $a \neq b$ and $m_{a,b} \ne \infty$.
        If we add the relation $a^2=1$ for all $a\in S$, the monoid given by this presentation is a group, $W$, and is called the \defit{Coxeter group} associated with $\mathcal{M}$. Then the Coxeter matrix (resp. the Artin-Tits monoid) is said to be of \defit{spherical type} if $W$ is finite.
    \end{definition}
    \begin{example}
            The Artin-Tits monoid of spherical type $B_3$ has Coxeter diagram \begin{tikzcd}
                a && c && b
                \arrow[no head, from=1-3, to=1-5]
                \arrow["4", no head, from=1-1, to=1-3]
            \end{tikzcd}
            which corresponds to the presentation $\pres{\{a,b,c\}}{acac=caca, cbc=bcb, ab=ba}^+$
    \end{example}
    A right-complemented presentation allows one to define a process called right-reversing \cite{dehornoy1992deux}: starting with a signed word $w \in (S\cup S^{-1})^*$, representing an element of the enveloping group of the monoid, one can try to write it as a fraction $u\fin{v}$ (where $\fin{v}$ is the notation for $v^{-1}$) with $u,v \in S^*$, by \enquote{pushing} all the negative letters to the right. Such writing is also called a \enquote{positive-negative path}. In practice, we replace a subword $\fin{s}t$ by $\rcb(s,t)\fin{\rcb(t,s)}$. Right-reversing is formally not an algorithm, it is a rewriting system, but since it is confluent, one can derive an algorithm from it (see \cite{Baader_Nipkow_1998} for the theory of rewriting systems). However, neither termination nor correctness of this algorithm is guaranteed in general. Such an algorithm for writing elements of the enveloping group of a monoid where the word problem is decidable would give a solution to the word problem in the group. This can be applied to the well-studied class of monoids called Garside monoids, introduced by Dehornoy and Paris in \cite{dehornoy_gaussian_1999}. That class includes the spherical type Artin-Tits monoids, where right-reversing behaves well. Multiple issues arise when trying to extend the use of right-reversing to the larger class of monoids with a \emph{finite Garside family}, which includes all Artin-Tits monoids \cite{dehornoy_garside_2015}.\par
    Another use of the right-reversing algorithm is to compute least common multiple. Starting with a word $\fin{u}v$, if right-reversing ends on a fraction $v'\fin{u'}$, then the element represented by $uv'=vu'$ is a common right-multiple of $u$ and $v$. Computing common right-multiples is crucial, for example in multifraction reduction \cite{dehornoy_multifraction_2017}, which is a new approach to the word problem in Artin-Tits groups in general. The computability of right-lcms with right-reversing was already known for a fixed monoid : if a bound on the length of the elements of a Garside family is given, one can write a correct algorithm to compute lcms, by computing the maximal number of right-reversing steps for finite computations. This is the argument for the computability of right-lcms in \cite{dehornoy_multifraction_2017}, Lemma 5.13.  However, we are not aware of any algorithm to compute this bound in general. This is a classic problem in computability theory, known as \emph{non-constructive algorithm proof}. The existence of the algorithm relies on the existence of a bound that is unkown. For Artin-Tits monoids, the existence of a finite Garside family \cite{dehornoy_garside_2015} is a consequence of the existence of finitely many \emph{low elements} in any Coxeter group, and we do not know if one can derive an algorithm from the proof \cite{dyer_small_2016}.\par
    In the case of Garside monoids, least common multiples always exist, and right-reversing detects them. In a more general framework, additional assumptions are needed. The correctness of computing least common multiples with right-reversing relies on the notion of completeness, for which different criteria were developed such as the $\theta$-cube condition. (see. \cite{dehornoy_groups_1997} and II.4.4 in \cite{dehornoy_foundations_2015}). However, termination is also tricky. After discussing some special cases where right-reversing is complete, Section II.4.3 in \cite{dehornoy_foundations_2015} ends with the following trichotomy when right-reversing is complete:
    \begin{enumerate}
        \item either right-reversing leads in finitely many steps to a positive-negative path;
        \item or one gets stuck at some step with a pattern $\fin{s}t$ such that $\rcb(s,t)$ and $\rcb(t,s)$ do not exist;
        \item or right-reversing can be extended endlessly reaching neither of the above two cases.
    \end{enumerate}\par
    Thus, even when right-reversing is complete, it does not provide an algorithm to decide if two elements have a common right-multiple, it only computes it when it exists. In practice, a possible  implementation is to just stop the algorithm after a large number of steps, and this is actually correct when a finite Garside family exists if the maximum number of steps is large enough, for the reason explained above.\par
    The easy case where right-reversing is terminating is when the signed word is decreasing in length at each step, which means that $\rcb(s,t)$ has length $0$ or $1$ for all $s,t \in S$. In that case, the right-complemented presentation is called \emph{short}. The notion of Garside family (also called set of basic elements) naturally arises in this context. The idea is to add $\rcb(s,t)$ and the lcm as new generators, for all $s,t$ successively to increase the generating set until it is closed under right-lcm and right-complement. Such a generating set is called a Garside family. It is a generalization of the notion of Garside monoid in the following sense: in a Garside monoid, generators share a common right-multiple called the Garside element, usually denoted by $\Delta$. Then the Garside family is the set of divisors of this element, i.e. the interval $[1,\Delta]$. The process of closing the atom set under right-lcm and right-complement ends in a finite number of steps only when there exists a finite Garside family, which is non-trivially the case for Artin-Tits monoids \cite{dehornoy_garside_2015}. Now, given a finite Garside family, one can compute in the monoid (as in Section 4 of \cite{dehornoy_algorithms_2014}). However, there was no algorithm computing least common multiples with only the assumption of the existence of a finite Garside family. Moreover, the natural algorithm that computes the smallest Garside family by closing the atom set under right-lcm and right-complement (\cite{dehornoy_algorithms_2014}, algorithm 2.23) may fail to terminate even if the smallest Garside family is finite because right-reversing can never terminate. So we might want to have a better understanding of what happens when right-reversing does not end.\par
    For example, when trying to close the atom set $\{a,b,c\}$ under right-lcm and right-complement in the Artin-Tits monoid of type $\widetilde{A_2} = \pres{\{a,b,c\}}{aba=bab, bcb=cbc, cac=aca}^+$ , the right-reversing of the word $\fin{bc}a$ never ends. Looking closely at what happens, the signed word $\fin{bc}a$ is reversed in three steps to the signed word $ac\fin{a}bc\fin{bacb}$. The subword $\fin{a}bc$ is then right-reversed to the signed word $bacb\fin{bc}a\fin{ac}$. Thus, the word $\fin{bc}a$ reappears in its own right-reversing. Inspired by the proof of the Kürzungsregeln \cite{brieskorn_artin-gruppen_1972}, we can conclude here that $bc$ and $a$ have no common right multiple. If we call such a pair "periodic", the natural question is to ask if infinite right-reversing is always caused by a periodic pair.
    \subsection{Results}
    Our main result answers this question in the classic setup where we assume that right-reversing is complete and there exists a finite Garside family. A pair $(u,v)$ is \emph{failing} if right-reversing of $\fin{u}v$ stops because of an undefined right-complement, and it is \emph{eventually periodic} if some periodic subword appears in the process.
        \begin{theorem}\label{theo:trichotomy_intro}
            Let $\pres{S}{R}^+$ be a right-complemented presentation of a monoid $M$ for which right-reversing is complete. We assume that $M$ has a finite Garside family $G$. Let $u, v$ be two words of $S^*$. Then one of the following holds:
            \begin{enumerate}
                \item\label{theo:trichotomy1_intro} $(u,v)$ is failing
                \item\label{theo:trichotomy2_intro} $(u,v)$ is eventually periodic
                \item\label{theo:trichotomy3_intro} $\erc{u}{v}$ is defined, i.e. there exist some positive words $u',v'$ on $S$ such that $\fin{u}v \rr^* v'\fin{u'}$.
            \end{enumerate}
        \end{theorem}
    In other word, we show that if right-reversing can be extended endlessly, then there is necessarily a cyclic pattern that can be detected. Stopping when a cycle is detected provides an algorithm for computing right-lcms in the right-complemented case based on right-reversing that does not need a bound on the length of the element in the minimal Garside family. This shows that right lcms is not computable only on a fixed monoid with a Garside family, it is computable given the presentation of a monoid with a finite Garside family as an input. It also allows one to write an algorithm for multifraction reduction without knowing a Garside family of the monoid \cite{dehornoy_multifraction_2017}. Moreover, this algorithm then provides a way of computing a minimal Garside family when it is finite. We implemented it and computed the Garside families for affine type Artin-Tits monoids $\widetilde{G_2}$ and $\widetilde{D_4}$.\par
    The paper is organized as follows. We begin with some preliminaries on arithmetics in monoids, with the notion of gcds, lcms, and noetherianity, based on \cite{dehornoy_foundations_2015}, II.2. Then, we present how computation can be done in a monoid given by generators and relations, introducing a new algorithm for lcms, inspired from the Kürzungsregeln of \cite{brieskorn_artin-gruppen_1972}, and we point out the similarity with the right-reversing algorithm of \cite{dehornoy_foundations_2015}. This new point of view reveals the notion of cyclic witness that can be detected and stop infinite computations. After recalling the notion of completeness, we formalize this notion of witness by introducing periodicity and eventual periodicity of right-reversing. We show that any infinite right-reversing is eventually periodic. We conclude with the main consequence : computability of the smallest Garside family when it exists. We also present computation results obtained with the implementation that can be found on the GitHub repository \url{https://github.com/domyrmininon/compute_gf}.
    \subsection{Thanks}
        This reflection arose from a master thesis initially focused on multifraction reduction theory of Dehornoy \cite{dehornoy_multifraction_2017} under the supervision of Yves Guiraud, whom I would like to thank. His reading recommendations were invaluable, and our weekly discussions, helped me formalize and clarify my thoughts, especially on the notion of rewriting system.
\section{Arithmetics in monoids}
    We start by describing the framework of right reduction. First with purely algebraic properties that will then be transposed into algorithmic results. The results of this section can be found in \cite{dehornoy_foundations_2015}, II.2. The idea is to have properties in the monoid that are analogous to the divisibility relation in the integers.
    \subsection{Divisibility relations in monoids}
        First, we recall the notion of cancellativity.
        \begin{definition}
            A monoid $M$ is \defit{left-cancellative} if, for all $a,b,c \in M$, $ca=cb$ implies $a=b$. It is \defit{right-cancellative} if, for all $a,b,c \in M$, $ac=bc$ implies $a=b$. It is \defit{cancellative} if it is both left and right-cancellative.
        \end{definition}
        The non-commutative law on monoids leads us to distinguish a left and a right divisibility relation.
        \begin{definition}
            Let $a,b \in M$. We say that $a$ is a \defit{left-divisor} of $b$, or equivalently that $b$ is a \defit{right-multiple} of $a$, if there exists $c \in M$ such that $b = ac$, and we write $a \ldiv b$.\\
            We say that $a$ is a \defit{right-divisor} of $b$, or equivalently that $b$ is a \defit{left-multiple} of $a$, if there exists $c \in M$ such that $b = ca$, and we write $a \ldivt b$.\\
            We write $\strldiv$ (resp. $\strldivt$) for the \defit{proper} left (resp. right) divisibility relation, that is, $ax = b$ (resp. xa = b) with $x$ non invertible.\\
            We denote by $\ldivset(a)$ (resp. $\rdivset(a)$) the set of left-divisors (resp. right-divisors) of $a$.
        \end{definition}
        The existence of a neutral element (resp. the associativity) in the monoid makes the relation $\ldiv$ and $\ldivt$ reflexive (resp. transitive), so they are both pre-orders on $M$. If $a$ and $b$ are such that $a \ldiv b$ and $b \ldiv a$, then there exist $c_1$, $c_2$ such that $b = ac_1$ and $a = bc_2$, and we can write $a = ac_1c_2$. If $M$ is left-cancellative, we can cancel $a$ and get $c_1c_2 = 1$, so $c_1$ and $c_2$ are invertible. If furthermore $M$ has no nontrivial invertible elements, we can conclude that $c_1 = c_2 = 1$, and we have $a = b$, so $\ldiv$ is a partial order. And similarly for $\ldivt$. We also have that $a\ldiv b$ implies $xa \ldiv xb$, and the converse is true if $M$ is left-cancellative, and similarly for $\ldivt$.\\
        Combining the two relations, we get the following definition.
        \begin{definition}
            Let $a,b \in M$. We say that $a$ is a \defit{factor} of $b$, if there exist $c,d \in M$ such that $cad=b$, and we write $a \factor b$. It is a \defit{proper factor} if at least one of $c$ and $d$ is non invertible, written $a \strfactor b$.
        \end{definition}
        Note that if $M$ has no nontrivial invertible elements, then $a \strfactor b$ (resp $a \strldiv b$ , resp. $a \strldivt b$) if and only if $a \factor b$ (resp. $a \ldiv b$, resp. $a \ldivt b$) and $a \ne b$.
    \subsection{Gcds and lcms}    
        \begin{definition}
            Let $M$ be a monoid, and $a,b,c \in M$. We say that $c$ is a \defit{left greatest common divisor} (left-gcd) of $a$ and $b$, if $c\ldiv a$ and $c\ldiv b$, and for all $d \in M$ such that $d\ldiv a$ and $d\ldiv b$, we have $d \ldiv c$.\\
            If $M$ is left-cancellative and has no nontrivial invertible elements, then the left-gcd is unique when it exists, as it is a greatest lower bound of $a$ and $b$ for the order $\ldiv$. We denote it by $c = a \gcd b$. Similarly, when $M$ is right-cancellative, we denote by $a \gcdt b$ the right-gcd of $a$ and $b$.
        \end{definition}
        This is enough to define gcd-monoids. However, the counterpart of gcd, namely, least common multiple will play a very important role in the theory of gcd-monoids. 
        \begin{definition}
            We say that $c$ is a \defit{right least common multiple} (right-lcm) of $a$ and $b$, if $a \ldiv c$ and $b \ldiv c$, and for all $d \in M$ such that $a \ldiv d$ and $b \ldiv d$, we have $c \ldiv d$. For the same reason, if $M$ is left-cancellative and has no nontrivial invertible elements, the right-lcm is unique when it exists, and we denote it by $c = a \lcm b$. Similarly, when $M$ is right-cancellative, we denote by $a \lcmt b$ the left-lcm of $a$ and $b$.
        \end{definition}
        When a right-lcm of $a$ and $b$ exists, we can write it $a \lcm b = aa'=bb'$. Then $a'$ (resp. $b'$) is called the \defit{right-complement} of $a$ in $b$ (resp. of $b$ in $a$), denoted by $(a\backslash b)$ (resp. $(b\backslash a)$), so we have $a \lcm b = a(a\backslash b) = b(b\backslash a)$.\\
        As the name suggests, gcds will always exist in gcd-monoids, but lcms may not exist. However, a weaker condition will be satisfied.
        \begin{definition}
            A cancellative monoid $M$ is said to have \defit{conditional right-lcms}, (resp. \defit{conditional left-lcms}) if for all $a,b \in M$, if $a$ and $b$ have a common right (resp. left) multiple, then they have a right-lcm (resp.  left-lcm). We say that $M$ has \defit{conditional lcms} if it has both conditional right-lcms and conditional left-lcms.
        \end{definition}
        Now, we can define the notion of gcd-monoid.
        \begin{definition}
            A \defit{gcd-monoid} is a monoid $M$ such that:
            \begin{enumerate}
                \item $M$ is cancellative,
                \item $M$ has no nontrivial invertible elements,
                \item For all $a,b \in M$, $a$ and $b$ have a left-gcd and a right-gcd.
            \end{enumerate}
        \end{definition}
        As said above, gcd-monoids have conditional lcms. The following properties have their equivalent when replacing left by right and vice versa.
        \begin{proposition}\label{prop:conditional_lcm}
            Let $M$ be a gcd-monoid. Then $M$ has conditional lcms.
        \end{proposition}
        First, we need a small lemma.
        \begin{lemma}\label{lemm:conditional_lcm}
            If $M$ is a gcd-monoid, then for all $a,b,c,d \in M$ such that $ad=bc$, then $ad$ is a right-lcm of $a$ and $b$ if and only if $c\gcdt d=1$.
        \end{lemma}
        \begin{proof}
            The proof is illustrated in Figure \ref{fig:lcm_gcd}. First, one can check that the left-gcd operator commutes with multiplication on the left, that is $a\gcd b = c$ if and only if $xa\gcd xb = xc$. Assume that $ad=bc=m$ is a right-lcm of $a$ and $b$. Let $x$ be a common right-divisor of $c$ and $d$, write, $d=d'x$ and $c=c'x$. Then, by cancellativity, $ad'=bc'$ is a common right-multiple of $a$ and $b$, and by definition of a right-lcm, $a \lcm b \ldiv ad' \ldiv ad = a\lcm b$ by assumption. So $ad'=ad$ and $x=1$. Conversely, assume $c\gcdt d=1$. Let $m'=ad'=bc'$ be a common right-multiple of $a$ and $b$. Let $d'' = d \gcd d'$ and $c'' = c \gcd c'$. Then, multiplying by $a$, we get $ad'' = ad \gcd ad' = bc \gcd bc' = bc''$, so we get a new commutative square $ad''=bc''=m''=m\gcd m'$. By the definition of a left-gcd, there exist $x,x'$ such that $m''x=m$ and $m''x'=m'$. The first expression can be rewritten in two ways: $ad''x=ad$ so $x \ldivt d$ and $bc''x=bc$ so $x \ldivt c$. By the definition of a right-gcd, $x=1$, and $m=m''$ so $m\ldiv m'$.
        \end{proof}
        \begin{figure}[ht]
            \begin{center}
                \begin{tikzpicture}\label{figimported:lcm_gcd}
                    \node[scale = 0.5] (P0) at (1, -1) {};
                    \node[scale = 0.5] (P1) at (3, -1) {};
                    \node[scale = 0.5] (P2) at (1, -3) {};
                    \node[scale = 0.5] (P3) at (3, -3) {};
                    \node[scale = 0.5] (P4) at (4, -4) {};
                    \node[scale = 0.5] (P5) at (5, -5) {};
                    \draw[-Stealth] (P0) edge [] node[auto]{$b$} (P1);
                    \draw[-Stealth] (P0) edge [] node[auto]{$a$} (P2);
                    \draw[-Stealth] (P1) edge [swap] node[auto]{$c$} (P3);
                    \draw[-Stealth] (P1) edge [color={rgb,255:red,214;green,92;blue,92}, bend right=-6] node[midway, fill=white]{${c''}$} (P4);
                    \draw[-Stealth] (P1) edge [bend right=-24] node[midway, fill=white]{${c'}$} (P5);
                    \draw[-Stealth] (P2) edge [] node[auto]{$d$} (P3);
                    \draw[-Stealth] (P2) edge [color={rgb,255:red,214;green,92;blue,92}, bend right=6] node[midway, fill=white]{${d''}$} (P4);
                    \draw[-Stealth] (P2) edge [bend right=24] node[midway, fill=white]{${d'}$} (P5);
                    \draw[-Stealth] (P4) edge [color={rgb,255:red,92;green,92;blue,214}] node[midway, fill=white]{$x$} (P3);
                    \draw[-Stealth] (P4) edge [color={rgb,255:red,92;green,92;blue,214}] node[midway, fill=white]{${x'}$} (P5);
                \end{tikzpicture}
            \end{center}
            \caption{Illustration of the proof of Lemma \ref{lemm:conditional_lcm}.}
            \label{fig:lcm_gcd}
        \end{figure}
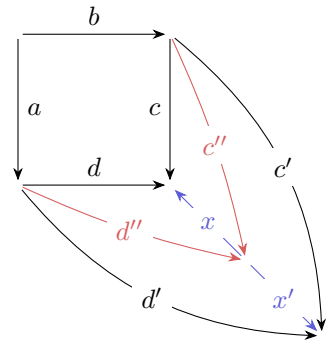
            
        \begin{proof}[ Proof of Proposition \ref{prop:conditional_lcm}]
            Let $ad=bc$ be a common right-multiple of $a$ and $b$. Put $e = c \gcdt d$, write $c=c'e$ and $d=d'e$, so that $c' \gcdt d' = 1$. Then by Lemma \ref{lemm:conditional_lcm}, $ac'=bd'$ is a right-lcm of $a$ and $b$.
        \end{proof}
            The vocabulary here is borrowed from arithmetics. However, the notion of right-lcm (resp. left-lcm) is, in the vocabulary of category theory, called a \defit{pushout} (resp. a \defit{pullback}) in the category with one object and elements of $M$ as morphisms. The usual habits to draw commutative diagram in category theory is very fruitful and works well in the framework of gcd-monoids. This leads to a property analogous to concatenation of pushout diagrams. See Figure \ref{fig:concat_square}.
        \begin{lemma}\label{lemm:concat_lcm}
            If $M$ is a gcd-monoid, then for all $a,b,c \in M$, the right-lcm $a\lcm bc$ exists if and only if $a\lcm b$ and $a'\lcm c$ exist, where $a'=(b\backslash a)$. In this case, we have $a\lcm bc = ab'c' = bca''$ where $b'=(a\backslash b), a''=(c\backslash a')$ and $c'=(a'\backslash c)$.
        \end{lemma}
        \begin{proof}          
                Assume that $m = a\lcm bc$ exists. Then $a$ and $b$ have a common right-multiple $m$, so by Proposition \ref{prop:conditional_lcm}, $a$ and $b$ have a right-lcm $a \lcm b = ab'=ba'$, and it left-divides $m$, so $m = ab'c'$. Moreover, $a''=(bc\backslash a)$ exists by assumption, and $ba'c'= m = bca''$, so cancelling $b$ gives $m'= a'c'=ca''$ so that $a'$ and $c$ have a common right-multiple. Again, they have a right-lcm $a' \lcm c = a'c'_1 = ca''_1$, and it left-divides $m'$, so $m' = a'c'_1x$, and cancelling $a$ gives $c'_1 \ldiv c'$. But then, $ab'c'_1=ba'c'_1= bca''_1$ is a right-multiple of $a$ and $bc$, so $ab'c'=m \ldiv ab'c'_1$, so left-cancelling $ab'$ gives $c' \ldiv c'_1$, hence $c' = c'_1 = (a'\backslash c)$ and $a'' = a''_1=(c\backslash a')$ follows. Conversely, assume that the right-lcm of $a$ and $b$ exists: $a\lcm b=ab'=ba'$, and that the right-lcm of $a'$ and $c$ exists: $a' \lcm c = a'c' = ca''$. Put $m = ab'c'=ba'c'=bca''$. Then $a$ and $bc$ have a common right-multiple $m$, so they have a right-lcm $m_1$. Then, the proof above shows that $m_1 =m$.
        \end{proof}
        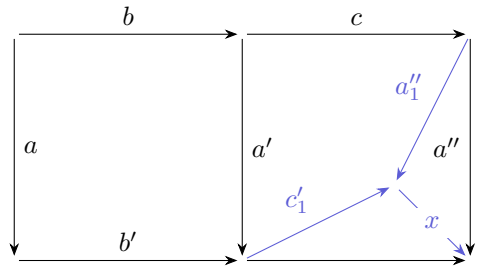
\begin{figure}[ht]
            \begin{center}
                    \begin{tikzpicture}\label{figimported:concat_square}
                        \node[scale = 0.5] (P0) at (1, -1) {};
                        \node[scale = 0.5] (P1) at (4, -1) {};
                        \node[scale = 0.5] (P2) at (7, -1) {};
                        \node[scale = 0.5] (P3) at (6, -3) {};
                        \node[scale = 0.5] (P4) at (1, -4) {};
                        \node[scale = 0.5] (P5) at (4, -4) {};
                        \node[scale = 0.5] (P6) at (7, -4) {};
                        \draw[-Stealth] (P0) edge [] node[auto]{$b$} (P1);
                        \draw[-Stealth] (P0) edge [] node[auto]{$a$} (P4);
                        \draw[-Stealth] (P1) edge [] node[auto]{$c$} (P2);
                        \draw[-Stealth] (P1) edge [] node[auto]{${a'}$} (P5);
                        \draw[-Stealth] (P2) edge [swap, color={rgb,255:red,92;green,92;blue,214}] node[auto]{${a''_1}$} (P3);
                        \draw[-Stealth] (P2) edge [swap] node[auto]{${a''}$} (P6);
                        \draw[-Stealth] (P3) edge [color={rgb,255:red,92;green,92;blue,214}] node[midway, fill=white]{$x$} (P6);
                        \draw[-Stealth] (P4) edge [] node[auto]{${b'}$} (P5);
                        \draw[-Stealth] (P5) edge [color={rgb,255:red,92;green,92;blue,214}] node[auto]{${c'_1}$} (P3);
                        \draw[-Stealth] (P5) edge [] node[auto]{} (P6);
                    \end{tikzpicture}    
            \end{center}
            \caption{Illustration of the proof of Lemma \ref{lemm:concat_lcm}. See the similarity with the concatenation of pushout diagrams in category theory.}
            \label{fig:concat_square}   
        \end{figure}
    \subsection{Noetherianity}
        The existence of gcds does not necessarily means that such gcds can be effectively computed. This is why the additional assumption of noetherianity is needed.
        \begin{definition}
            A monoid $M$ is called \defit{left-noetherian} if the relation $\strldiv$ is well-founded, that is, there is no infinite decreasing sequence for $\strldiv$. It is called \defit{right-noetherian} if the relation $\strldivt$ is well-founded, and \defit{noetherian} if the relation $\strfactor$ is well-founded.
        \end{definition}
        When $M$ is cancellative, the notion of noetherianity can be described by the existence of a noetherianity witness a map from the monoid to the collection of ordinals \textbf{Ord}:
        \begin{definition}
            A map $\lambda:M \to$ \textbf{Ord} is called a \defit{left-noetherianity witness} (resp. \defit{right-noetherianity witness}), if for all $a,b \in M$, $a \strldiv b$ (resp. $a \strldivt b$) implies $\lambda(a) < \lambda(b)$. It is called \defit{sharp} if, additionally, $a\ldiv b$ (resp. $a \ldivt b$) implies $\lambda(a) \le \lambda(b)$.
        \end{definition}
        The following proposition has its right counterpart.
        \begin{proposition}
            Let $M$ be a cancellative monoid. Then the following are equivalent:
            \begin{enumerate}
                \item\label{prop:noetherian} $M$ is left-noetherian.
                \item\label{prop:noetherian_witness} $M$ has a left-noetherianity witness.
                \item\label{prop:noetherian_sharp_witness} $M$ has a sharp left-noetherianity witness.
            \end{enumerate}
        \end{proposition}
        \begin{proof}
            \eqref{prop:noetherian} $\Leftrightarrow$ \eqref{prop:noetherian_witness} is a basic result of set theory (see \cite{levy_basic_1979}).\\
            \eqref{prop:noetherian_sharp_witness} $\Rightarrow$ \eqref{prop:noetherian_witness}: a sharp left-noetherianity witness is by definition a left-noetherianity witness.\\
            \eqref{prop:noetherian_witness} $\Rightarrow$ 
            \eqref{prop:noetherian_sharp_witness}: Let $\lambda$ be the map given by \eqref{prop:noetherian_witness}. Let $M^{\text{inv}}$ be the set of invertible elements of $M$. Let $\mu:M\to$\textbf{Ord} be the map defined by $\mu(a) = \min_{u\in M^{\text{inv}}}(\lambda(au))$. This is well defined since any nonempty set of ordinals has a minimum, and there is an $a'=au_a$ such that $\mu(a)=\lambda(a')$. Let $a \ldiv b$, i.e. $ac = b$ with $c\in M$. If $c$ is non invertible,  then there exists $a'=au_a$ and $b'=bu_b$ such that $\mu(a) = \lambda(a')$ and $\mu(b) = \lambda(b')$. But $a'u_a^{-1}cu_b=acu_b=bu_b=b'$ and $u_a^{-1}cu_b$ is not invertible, so $a' \strldiv b'$, and by assumption \eqref{prop:noetherian_witness}, $\lambda(a') < \lambda(b')$. So $\mu(a) < \mu(b)$. Finally, if $c$ is invertible, then $\mu(a)=\mu(b)$. So $\mu$ is a sharp left-noetherianity witness.            
        \end{proof}
        The notion of witness can be used to define a stronger notion of noetherianity.
        \begin{definition}
            A monoid $M$ is called \defit{strongly noetherian} if there exists a map $\lambda:M \to \N$ such that for all $a,b \in M, \lambda(ab) \ge \lambda(a)+\lambda(b)$ and $\lambda(a) > 0$ for every non invertible element $a\in M.$
        \end{definition}
        Strong noetherianity implies noetherianity, as we simply constrained the values of the witness to be integers instead of any ordinals.\\       Cancellability also allows to identify noetherianity with left and right-noetherianity.
        \begin{proposition}
            Let $M$ be a cancellative monoid. Then $M$ is noetherian if and only if it is right-noetherian and left-noetherian.
        \end{proposition}
        \begin{proof}
            Assume $\strfactor$ is well-founded. Then any infinite decreasing sequence for $\strldiv$ (resp. $\strldivt$) is also an infinite decreasing sequence for $\strfactor$ so it is not possible and $M$ is left (resp. right) -noetherian. Conversely, assume that $M$ is left and right-noetherian. Let $(a_n)_{n\in \N}$ be a decreasing sequence for $\strfactor$. Then $a_n=d_na_{n+1}c_n$ for all $n$ and at least one of $d_n$ and $c_n$ is not invertible. Then at least one of the sequences $(d_n)_{n\in \N}$ or $(c_n)_{n\in \N}$ contains infinitely many non invertible elements. Assume for example that it is $(c_n)$. Define $a'_i=d_0d_1\ldots d_{i-1}a_i$ for all $i\in \N$. Then $a'_i=a'_{i+1}c'_i$ for all $i$, and since $c'_i$ is not invertible infinitely many times, one can extract a decreasing sequence from $(a'_i)$ for $\strldiv$. Symmetrically, if $(d_n)$ contains infinitely many non invertible elements, we can extract a decreasing sequence from $(a'_i)$ for $\strldivt$. In both cases, we get a contradiction with the left or right-noetherianity of $M$. So such sequence cannot exist and $M$ is noetherian.
        \end{proof}
        Noetherianity allows to characterize gcd-monoids using the conditional lcms property.
        \begin{proposition}\label{prop:conditional_lcm_to_gcd}
            Let $M$ be a cancellative and noetherian monoid. Then $M$ is a gcd-monoid if and only if it has conditional lcms.
        \end{proposition}
        This is the consequences of the following two lemmas, and their symmetric versions. The first one is a mixed interaction between left-cancellativity and right-noetherianity.
        \begin{lemma}\label{lemm:mixed_noetherian}
            Let $M$ be a left-cancellative and right-noetherian monoid and $a\in M$. Then, the set $\ldivset(a)$ of left-divisors of $a$ contains no infinite increasing sequence for $\strldiv$. In particular, any nonempty subset of $\ldivset(a)$ has a maximal element for $\strldiv$.
        \end{lemma}
        \begin{proof}
            Assume by contradiction that we can build an infinite increasing sequence $c_0 \strldiv c_1 \strldiv \ldots$ of elements of $\ldivset(a)$. Then we can write $c_{i+1} = c_i d_i$ for some non invertible element $d_i \in M$. But $c_i \in \ldivset(a)$, so, in particular, there exists $x_i$ such that $c_ix_i = a$. Then, we have $c_id_ix_{i+1}=c_{i+1}x_{i+1}=a=c_{i}x_i$ and by left-cancellativity, $d_ix_{i+1}=x_i$. So $(x_i)$ is a decreasing sequence for $\strldivt$, and this is impossible because $M$ is right-noetherian.
        \end{proof}
        The second lemma is a consequence of having conditional right-lcms.
        \begin{lemma}\label{lemm:mixed_lcm}
            Let $M$ be a monoid with conditional right-lcms. Then for all $a \in M$, and $b,c \in \ldivset(a)$, the right-lcm $b \lcm c$ exists and is in $\ldivset(a)$.
        \end{lemma}
        \begin{proof}
            Let $b,c \in \ldivset(a)$. Then $b$ and $c$ have a common right-multiple $a$, so by the conditional right-lcms property, they have a right-lcm $m = b \lcm c$ and it left-divides $a$.
        \end{proof}
        Now we can prove the proposition.
        \begin{proof}[Proof of Proposition \ref{prop:conditional_lcm_to_gcd}]
            We already know that if $M$ is a gcd-monoid, then it has conditional right-lcms by Proposition \ref{prop:conditional_lcm}. Assume now that $M$ has conditional right-lcms. Let $a,b \in M$, and let $\ldivset(a,b)=\ldivset(a) \cap \ldivset(b)$ be the set of common left-divisors of $a$ and $b$. By Lemma \ref{lemm:mixed_noetherian}, $\ldivset(a,b)$ has a maximal element $c$. To prove that it is a maximum, it is enough to show that it is the unique maximal element of $\ldivset(a,b)$. Assume that $c'$ is another maximal element of $\ldivset(a,b)$. Then by Lemma \ref{lemm:mixed_lcm}, $c$ and $c'$ have a right-lcm $m = c \lcm c'$, and $m \in \ldivset(a)\cap \ldivset(b)$. Thus, by maximality of $c$ (resp. $c'$), we have $m=c$ (resp. $m=c'$). So $c=c'$ is the left-gcd of $a$ and $b$. The symmetric argument shows that $M$ has right-gcds.
        \end{proof}
        \begin{remark}
            The previous proposition is the  most common way to prove that a monoid is a gcd-monoid. It is how Artin-Tits monoids are proved to be gcd-monoids at first in \cite{brieskorn_artin-gruppen_1972}, and it applies to some particular class of presented monoids, which is the subject of the next section.
        \end{remark}

\section{Computing in presented monoids}
    \subsection{Monoid presentation}
        We recall the definition of a monoid presentation. In the following, a finite set of generators $S$ is fixed. Elements of $S$ are called \defit{letters} and elements of $S^*$ are called \defit{words}. The empty word is denoted by $\varepsilon$. The length of a word $w$ is denoted by $\wl{w}_S$. 
        \begin{definition}
            A \defit{relation} is a pair $(u,v)$ where $u,v$ are words of $S^*$. A \defit{congruence} $\eqeq$ on a monoid $M$ is an equivalence relation compatible with the product, i.e. satisfying, for all $u,v,w \in M$, if $u\eqeq v$, then $wu\eqeq wv$ and $uw \eqeq vw$. The equivalence class under $\eqeq$ form the quotient monoid $M = S^*{/}{\eqeq}$. If $R$ is a set of relations, there exists a smallest congruence $\eqeq$ containing $R$, that is, such that for all $(u,v) \in R, u\eqeq v$. We say that the quotient monoid $S^*{/}{\eqeq}$ has the presentation $\pres{S}{R}^+$. In the following, we will always consider finite presentations, i.e. both $S$ and $R$ are finite.
        \end{definition}
        \begin{example}
            The monoid $\N^2$ has the presentation $\pres{a,b}{ab=ba}^+$.
        \end{example}
        Matching the notation of \cite{brieskorn_artin-gruppen_1972}, for two words $X,Y \in S^*$ we will note $X\weq Y$ when $X,Y$ are equal as words of $S^*$ and $X \meq Y$ if $X$ and $Y$ represent the same element of $M$, i.e. if $X \eqeq Y$ with $\eqeq$ the smallest congruence that contains $R$. (This relation is written $\equiv_R$ in \cite{dehornoy_foundations_2015}).
        \begin{definition}
            A presentation $\pres{S}{R}^+$ is \defit{homogeneous} if all relation are length preserving, that is, for all $(u,v) \in R, \wl{u}_S = \wl{v}_S$.
        \end{definition}
        \begin{remark}\label{rem:homo_strong}
        The length function $u \mapsto \wl{u}_S$ is stable under $\meq$ when the presentation is homogeneous, so it is well defined on the quotient $M=S^*{/}{\meq}$, hence it gives a strong noetherianity witness for $M$.
        \end{remark}
        We recall the notion of derivation.
        \begin{definition}
            Let $R$ be a set of relations on $S$.\\
            If $X,Y \in S^*$ are words, we will say that $Y$ is \defit{derived} from $X$ in one step, if there exist some words $X_1, X_2, r_1, r_2$ such that $X \weq X_1r_1X_2$, $Y \weq X_1r_2X_2$ and $(r_1,r_2) \in R$ or $(r_2,r_1) \in R$. A word $Y$ is \defit{derived} from $X$ in $n$ steps if there exists a sequence of words $X=W_0,W_1,\ldots,W_n=Y$ such that $W_{i+1}$ is derived from $W_i$ in one step. The relation $X \meq Y$ is in fact the same as "$Y$ is derived from $X$ in $n$ steps for some $n$".
        \end{definition}
        The word problem with respect to a monoid presentation $\pres{S}{R}^+$ is the problem to decide, given to word $w_1,w_2 \in S^*$ if they are equal in the monoid, i.e. if $w_1 \meq w_2$. It is decidable when the presentation is homogeneous.
        \begin{proposition}\label{prop:homo_word}
            Let $\pres{S}{R}^+$ be a homogeneous finite presentation of a monoid $M$. Then the word problem with respect to $\pres{S}{R}^+$ is decidable.
        \end{proposition}
        \begin{proof}
            Let's show that the $\meq$-classes are finite. Let $w\in S^*$ and $a\in M$ Since $\pres{S}{R}^+$ is homogeneous, any $w'$ that represents $a$ verifies $\wl{w'}_S=\wl{w}_S$. Then, the $\meq$-class of $w$ is included in $S^{\wl{w}_S}$ which is finite. Now, given $w,w'$ such that $\wl{w}_S=\wl{w'}_S=n=p$, one can check, since $R$ is finite if $w'$ is derived from $w$ in $n$ step with some $n< |S|^p$.
        \end{proof}
        Decidability of the word problem in the homogeneous case gives us decidability of the left and right-divisibility relations, because divisor sets can be computed. Gcds can also be computed since the number of divisors is finite. It leaves the problem of lcms.
        
    \subsection{Right-complemented presentation}
        The next definition follows the framework of \cite{dehornoy_foundations_2015} II.4.
        \begin{definition}
            A monoid presentation ${\pres{S}{R}^+}$ is called a \defit{right-complemented presentation} if $R$ contains no $\varepsilon$-relation (that is, no relation $w=\varepsilon$ with $w$ nonempty), no relation $s\ldots = s\ldots$ with $s$ in $S$, and at most one relation $s\ldots = t\ldots$ for $s\ne t$ in $S$. A right-complemented presentation is always associated with a syntactic right-complement, that is the partial map $\rcb$ from $S^2$ to $S^*$ such that $\rc{s}{s}= \varepsilon$ for all $s$ in $S$, and $\rc{s}{t}$ and $\rc{t}{s}$ are defined if and only if there is a relation $s\ldots=t\ldots$ in $R$, and, in this case, the relation is $s\rc{s}{t}=t\rc{t}{s}$. A right-complemented presentation is called \defit{short} if $\rc{s}{t} \in S \cup \{\varepsilon\}$ for all $s,t$ in $S$.
        \end{definition}
        Note that, in light of the previous section, a right-complemented presentation seems to give right least common multiples for elements of $S$. This will be the case for Artin-Tits monoids, but it's not always true. Also, knowing right-lcm for generators does not necessarily means that right-lcm computation is possible for any pair of elements. We investigate this problem in what follows.

    \subsection{Reduction rule}
        The fact that the right-complemented presentation really gives the right-lcm of generators was proved for Artin-Tits monoids by Brieskorn and Saito in \cite{brieskorn_artin-gruppen_1972} (Lemma 2.1). We present this very important statement called the reduction rule (Kürzungsregeln) and some of its useful consequences.
        \begin{definition}[The reduction rule]\label{defi:reduction_rule}
            If $\pres{S}{R}^+$ is a right-complemented presentation associated with the syntactic right-complement $\rcb$, we say that $\pres{S}{R}^+$ satisfies the \defit{reduction rule} if:\\
            Let $X,Y \in S^*$ , and $a,b \in S$ such that $aX \meq bY$. Then, $\rc{a}{b}$ and $\rc{b}{a}$ are defined and there exists a word $W$ such that:
            \[
                X \meq \rc{a}{b}W \hspace{3ex}\text{ and }\hspace{3ex} Y \meq \rc{b}{a}W
            \]
        \end{definition}
        We start by giving some immediate consequences of this rule.
        \begin{lemma}\label{lemm:reduction}
            If $M = \pres{S}{R}^+$ is a right-complemented presentation that satisfies the reduction rule, then:
            \begin{enumerate}
                \item\label{reduc:Slcm} For all $a,b \in S$, $a\rc{a}{b}\meq b\rc{b}{a}$ is the right-lcm of $a$ and $b$ if it is defined, and $a,b$ have no common right-multiple otherwise.
                \item\label{reduc:lcancel} $M$ is left-cancellative
                \item\label{reduc:conlcm} $M$ has conditional right-lcms, i.e. any two elements of $M$ having a common right-multiple have a right-lcm.
            \end{enumerate}
        \end{lemma}
        \begin{proof}
            \eqref{reduc:Slcm}: is a restatement of the fact that the reduction rule holds.\\
            \eqref{reduc:lcancel}: Let $X,Y,Z \in S^*$. We prove by induction on $\wl{Z}_S$ that $ZX\meq ZY$ implies $X\meq Y$. If $\wl{Z}_S=0$, then $X\weq Y$ so $X\meq Y$. Assume $\wl{Z}_S = l \ge 1$. Write $Z = aZ'$ with $a \in S$ and $Z' \in S^*$ with $\wl{Z'}=l-1$. We have $aZ'X\meq aZ'Y$. By the reduction rule, there exists a word $W$ such that $Z'X \meq \rc{a}{a}W$ and $Z'Y \meq \rc{a}{a}W$. But $\rc{a}{a}=\varepsilon$, so $Z'X \meq W$ and $Z'Y \meq W$, so $Z'X\meq Z'Y$, and the induction hypothesis applies because $\wl{Z'}\le l-1$. So $X\meq Y$. If $x,y,z \in M$ are such that $zx =zy$, then if we take $X,Y,Z \in S^*$ that represent $x,y,z$ respectively, we have $ZX\meq ZY$, so by the previous result, $X\meq Y$, hence $x=y$.\\ 
            \eqref{reduc:conlcm}: Let $u,v \in S^*$ such that there exist $W_0, W_1 \in S^*$ verifying $uW_0\meq vW_1$. We prove by induction on $l=\wl{uW_0}=\wl{vW_1}$ that $u$ and $v$ have a common right-multiple that left-divides $uW_0$. (This for all $l$ will imply that $u$ and $v$ have a right-lcm).\\
            If $l=0$ then $u$ and $v$ are empty and $\varepsilon$ is a right-lcm. Assume $l\ge 1$. If $u$ or $v$ is empty, the right-lcm is represented by the non empty element of the pair $(u,v)$. Assume $u$ and $v$ are both non-empty and write $u \weq a\tilde{u}$ and $v \weq b\tilde{v}$. The reduction rule will allow us to build the diagram on Figure \ref{fig:reduction_rule}. By the reduction rule, there exists a word $W_2 \in S^*$ such that $\tilde{u}W_1 \meq \rc{a}{b}W_2$ and $\rc{b}{a}W_2\meq \tilde{v}W_0$. By the induction hypothesis, $\tilde{u}$ and $\rc{a}{b}$ have a common right-multiple, we denote it by $\tilde{u}v'_1=\rc{a}{b}\tilde{u}_1$, and it is a left-divisor of $\tilde{u}W_1\meq\rc{a}{b}W_2$, i.e. there exist a word $W_4$ such that $\rc{a}{b}\tilde{u}_1W_4 \meq \rc{a}{b}W_2$. Successive use of the reduction rule allows to simplify on the left and we get $\tilde{u}_1W_4\meq W_2$. Symmetrically, there is a common right-multiple of $\tilde{v}$ and $\rc{b}{a}$ which is a left-divisor of $\rc{b}{a}\tilde{v}_1W_3\meq \tilde{v}W_0$, so there exist words $W_3, u'_1$ and $\tilde{v}_1$ such that $\rc{b}{a}\tilde{v}_1W_3\meq \rc{b}{a}W_2$ and we can simplify to get $\tilde{v}_1W_3\meq W_2$. Then $\tilde{u}_1$ and $\tilde{v}_1$ have a common right-multiple, namely $W_2$, and we can apply the induction hypothesis to get that $\tilde{u}_1$ and $\tilde{v}_1$ have a common right-multiple, that we denote by $\tilde{u}_1v'_2\meq\tilde{v}_1u'_2$, which is a left-divisor of $W_2$, i.e. there exists a word $W_5$ such that $\tilde{u}_1v'_2W_5\meq W_2$. We put $u'=u_1'u_2'$ and $v'=v_1'v_2'$ and we have $uv' = a\tilde{u}v_1'v_2'\meq a\rc{a}{b}\tilde{u}_1v_2'\meq b\rc{b}{a}\tilde{v}_1\tilde{v}_2\meq b\tilde{v}u_1'u_2'=vu'$ and it is a left-divisor of $uW_0\meq vW_1$.
            \end{proof}
            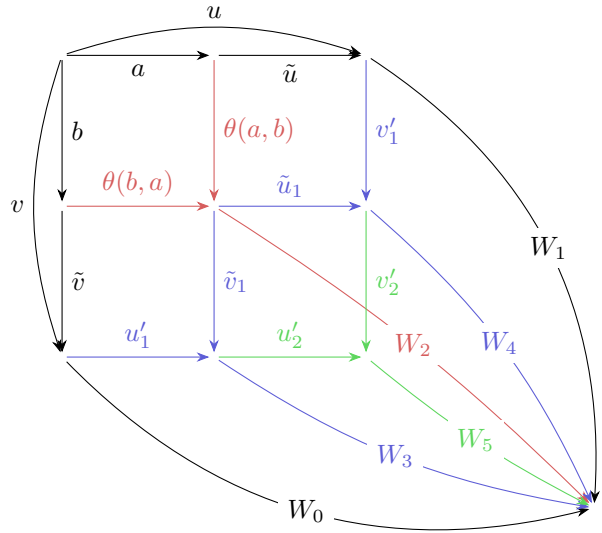
\begin{figure}
            \[\begin{tikzpicture}\label{figimported:reduction_induction}
                \node[scale = 0.5] (P0) at (1, -1) {};
                \node[scale = 0.5] (P1) at (3, -1) {};
                \node[scale = 0.5] (P2) at (5, -1) {};
                \node[scale = 0.5] (P3) at (1, -3) {};
                \node[scale = 0.5] (P4) at (3, -3) {};
                \node[scale = 0.5] (P5) at (5, -3) {};
                \node[scale = 0.5] (P6) at (1, -5) {};
                \node[scale = 0.5] (P7) at (3, -5) {};
                \node[scale = 0.5] (P8) at (5, -5) {};
                \node[scale = 0.5] (P9) at (8, -7) {};
                \draw[-Stealth] (P0) edge [swap] node[auto]{$a$} (P1);
                \draw[-Stealth] (P0) edge [bend right=-18] node[auto]{$u$} (P2);
                \draw[-Stealth] (P0) edge [] node[auto]{$b$} (P3);
                \draw[-Stealth] (P0) edge [swap, bend right=18] node[auto]{$v$} (P6);
                \draw[-Stealth] (P1) edge [swap] node[auto]{${{\tilde{u}}}$} (P2);
                \draw[-Stealth] (P1) edge [color={rgb,255:red,214;green,92;blue,92}] node[auto]{${{\theta(a,b)}}$} (P4);
                \draw[-Stealth] (P2) edge [color={rgb,255:red,92;green,92;blue,214}] node[auto]{${{v'_1}}$} (P5);
                \draw[-Stealth] (P2) edge [bend right=-30] node[midway, fill=white]{${{W_1}}$} (P9);
                \draw[-Stealth] (P3) edge [color={rgb,255:red,214;green,92;blue,92}] node[auto]{${{\theta(b,a)}}$} (P4);
                \draw[-Stealth] (P3) edge [] node[auto]{${{\tilde{v}}}$} (P6);
                \draw[-Stealth] (P4) edge [color={rgb,255:red,92;green,92;blue,214}] node[auto]{${\tilde{u}_1}$} (P5);
                \draw[-Stealth] (P4) edge [color={rgb,255:red,92;green,92;blue,214}] node[auto]{${\tilde{v}_1}$} (P7);
                \draw[-Stealth] (P4) edge [color={rgb,255:red,214;green,92;blue,92}, bend right=-6] node[midway, fill=white]{${{W_2}}$} (P9);
                \draw[-Stealth] (P5) edge [color={rgb,255:red,92;green,214;blue,92}] node[auto]{${{v'_2}}$} (P8);
                \draw[-Stealth] (P5) edge [color={rgb,255:red,92;green,92;blue,214}, bend right=-12] node[midway, fill=white]{${W_4}$} (P9);
                \draw[-Stealth] (P6) edge [color={rgb,255:red,92;green,92;blue,214}] node[auto]{${{u'_1}}$} (P7);
                \draw[-Stealth] (P6) edge [bend right=30] node[midway, fill=white]{${{W_0}}$} (P9);
                \draw[-Stealth] (P7) edge [color={rgb,255:red,92;green,214;blue,92}] node[auto]{${{u'_2}}$} (P8);
                \draw[-Stealth] (P7) edge [color={rgb,255:red,92;green,92;blue,214}, bend right=12] node[midway, fill=white]{${W_3}$} (P9);
                \draw[-Stealth] (P8) edge [color={rgb,255:red,92;green,214;blue,92}, bend right=6] node[midway, fill=white]{${{W_5}}$} (P9);
            \end{tikzpicture}\]
            \caption{Illustration of the proof of Lemma \ref{lemm:reduction} \eqref{reduc:conlcm}.} 
            \label{fig:reduction_rule}
            \end{figure}
            The fact that the monoid has conditional right-lcms is not fully satisfying. We would like to have an algorithm that computes the right-lcm of two elements when it exists, and tells us if it doesn't. From the proof of Lemma \ref{lemm:reduction} \eqref{reduc:conlcm} the result above can in fact be derived to an algorithm that computes the right-lcm when it exists. We just start with the assumption common multiple exists, and use the reduction rule successively to find $u'$ and $v'$. If some relation $a\ldots W_i\meq b\ldots W_j$ appears but $\rc{a}{b}$ is undefined, the algorithm stops with no lcm found. The process is illustrated in the following example.
            \begin{example}\label{ex:reduction_rule}
            We start with the initial words $ac$ and $bc$, representing elements of the Artin-Tits monoid of type $A_3$:\begin{tikzcd}
                a && c && b
                \arrow[no head, from=1-1, to=1-3]
                \arrow[no head, from=1-3, to=1-5]
            \end{tikzcd}. The process starts with the assumption that there exists a common multiple $acU=bcV$ and deduces from the reduction rule the relations in the graph illustrated in figure \ref{fig:reduction_rule}. We can add additional edges and vertices to the graph, by going down from the vertex labelled by an equation $W_i = y_1\ldots y_nW_j$ by simplifying each letter of $y$ one by one, labelling the vertices not by an equation but simply by the words $y_r\ldots y_nW_j$, until we reach the vertex labelled by $W_j$ and were it can be merged with another vertex where $W_j$ appears. The graph obtained is a Van Kampen diagram whose boundary gives a commutative square $uv'\meq v'u$. This second process is illustrated in figure \ref{fig:reduction_rule_completed}.
            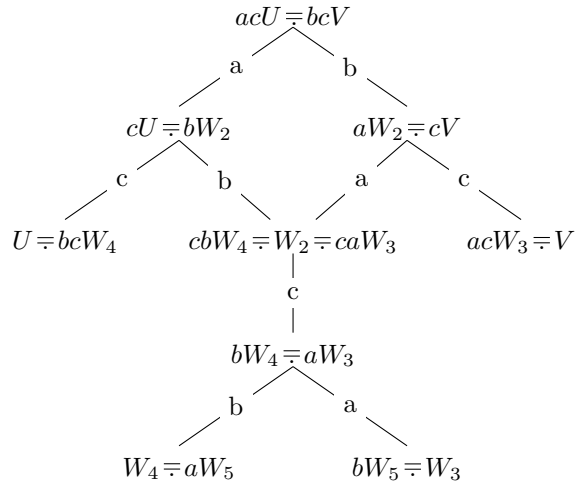
\begin{figure}
            \[\begin{tikzpicture}
                \node at (-1.5, -0.0) {$acU \meq bcV$};
                \node at (-3.0, -1.5) {$cU \meq bW_{2}$};
                \draw[color=black] (-3.0,-1.2) -- (-1.5,-0.15000000000000002) node [midway, fill= white] {a};
                \node at (0.0, -1.5) {$aW_{2} \meq cV$};
                \draw[color=black] (0.0,-1.2) -- (-1.5,-0.15000000000000002) node [midway, fill= white] {b};
                \node at (-4.5, -3.0) {$U \meq bcW_{4}$};
                \draw[color=black] (-4.5,-2.7) -- (-3.0,-1.65) node [midway, fill= white] {c};
                \node at (-1.5, -3.0) {$cbW_{4} \meq W_{2} \meq caW_{3}$};
                \draw (-1.8,-2.7) -- (-3.0,-1.65) node [midway, fill= white] {b};
                \draw (-1.2,-2.7) -- (0.0,-1.65) node [midway, fill= white] {a};
                \node at (1.5, -3.0) {$acW_{3} \meq V$};
                \draw[color=black] (1.5,-2.7) -- (0.0,-1.65) node [midway, fill= white] {c};
                \node at (-1.5, -4.5) {$bW_{4} \meq aW_{3}$};
                \draw[color=black] (-1.5,-4.2) -- (-1.5,-3.15) node [midway, fill= white] {c};
                \node at (-3.0, -6.0) {$W_{4} \meq aW_{5}$};
                \draw[color=black] (-3.0,-5.7) -- (-1.5,-4.65) node [midway, fill= white] {b};
                \node at (0.0, -6.0) {$bW_{5} \meq W_{3}$};
                \draw[color=black] (0.0,-5.7) -- (-1.5,-4.65) node [midway, fill= white] {a};
            \end{tikzpicture}\]  
        \caption{Graph obtained by applying the reduction rule to the initial assumption $acU\meq bcV$ in type $A_3$. When going down in the graph, going through an edge labelled by $x$, the word in the bottom vertex are obtained by simplifying the first letter $x$ of the word in the top vertex.}
        \label{fig:reduction_rule}
        \end{figure}
        \end{example}
        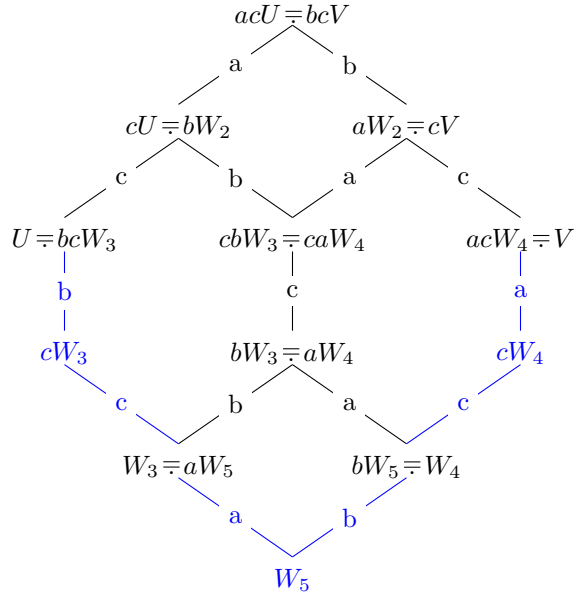
\begin{figure}
        \[\begin{tikzpicture}
            \node[color=black] at (-1.5, -0.0) {$acU \meq bcV$};
            \node[color=black] at (-3.0, -1.5) {$cU \meq bW_{2}$};
            \node[color=black] at (0.0, -1.5) {$aW_{2} \meq cV$};
            \node[color=black] at (-4.5, -3.0) {$U \meq bcW_{3}$};
            \node[color=black] at (-1.5, -3.0) {$cbW_{3} \meq caW_{4}$};
            \node[color=black] at (1.5, -3.0) {$acW_{4} \meq V$};
            \node[color=blue] at (-4.5, -4.5) {$cW_{3}$};
            \node[color=black] at (-1.5, -4.5) {$bW_{3} \meq aW_{4}$};
            \node[color=blue] at (1.5, -4.5) {$cW_{4}$};
            \node[color=black] at (-3.0, -6.0) {$W_{3} \meq aW_{5}$};
            \node[color=black] at (0.0, -6.0) {$bW_{5} \meq W_{4}$};
            \node[color=blue] at (-1.5, -7.5) {$W_{5}$};
            \draw[color=black] (-3.0,-1.2) -- (-1.5,-0.15000000000000002) node [midway, fill= white] {a};
            \draw[color=black] (0.0,-1.2) -- (-1.5,-0.15000000000000002) node [midway, fill= white] {b};
            \draw[color=black] (-4.5,-2.7) -- (-3.0,-1.65) node [midway, fill= white] {c};
            \draw[color=black] (1.5,-2.7) -- (0.0,-1.65) node [midway, fill= white] {c};
            \draw[color=black] (-1.5,-2.7) -- (-3.0,-1.65) node [midway, fill= white] {b};
            \draw[color=black] (-1.5,-2.7) -- (0.0,-1.65) node [midway, fill= white] {a};
            \draw[color=black] (-1.5,-4.2) -- (-1.5,-3.15) node [midway, fill= white] {c};
            \draw[color=black] (-3.0,-5.7) -- (-1.5,-4.65) node [midway, fill= white] {b};
            \draw[color=black] (0.0,-5.7) -- (-1.5,-4.65) node [midway, fill= white] {a};
            \draw[color=blue] (-4.5,-4.2) -- (-4.5,-3.15) node [midway, fill= white] {b};
            \draw[color=blue] (-3.0,-5.7) -- (-4.5,-4.65) node [midway, fill= white] {c};
            \draw[color=blue] (-1.5,-7.2) -- (-3.0,-6.15) node [midway, fill= white] {a};
            \draw[color=blue] (1.5,-4.2) -- (1.5,-3.15) node [midway, fill= white] {a};
            \draw[color=blue] (0.0,-5.7) -- (1.5,-4.65) node [midway, fill= white] {c};
            \draw[color=blue] (-1.5,-7.2) -- (0.0,-6.15) node [midway, fill= white] {b};
        \end{tikzpicture}\]
        \caption{Completing the graph from figure \ref{fig:reduction_rule} to get a Van Kampen diagram.}
        \label{fig:reduction_rule_completed}
        \end{figure}
        This algorithm, that works only on homogeneous presentations, is very similar (and equivalent) to the right-reversing algorithm of \cite{dehornoy_foundations_2015}.

        \begin{remark}\label{remark:termination}
            The construction of the graph described above does not necessarily end. We know by the proof of Lemma \ref{lemm:reduction} \eqref{reduc:conlcm}, that it ends and gives the right-lcm of $(u,v)$ if the pair $(u,v)$ has a common right-multiple. Conversely, if it ends, it gives a right-common multiple of $(u,v)$, and it is necessarily the right-lcm. Now it can also fail if some relation $a\ldots W_i\meq b\ldots W_j$ appears in the graph but $\rc{a}{b}$ is undefined, and it can also continue indefinitely.\\
            In the first case, we can stop the algorithm and conclude that the pair $(u,v)$ has no common right-multiple.\\
            But to fully decide if the pair $(u,v)$ has a common right-multiple, we need a way to stop the algorithm when the graph is infinite.
        \end{remark}

        \begin{example}\label{ex:termination}
            We reiterate the process of the previous example with the same words $(ac,bc)$ but this time with the Artin-Tits monoid of type $\widetilde{A_2}$ that has presentation $\pres{\{a,b,c\}}{aba=bab, bcb=cbc, cac=aca}$:
            This process is illustrated in figure \ref{fig:non_termination_example}.
            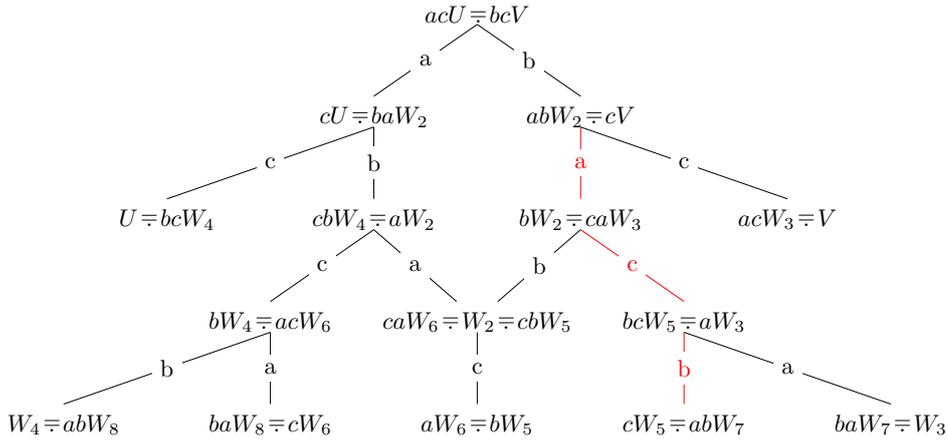
\begin{figure}
            \resizebox{\linewidth}{!}{%
                \begin{tikzpicture}
                    \node at (-1.5, -0.0) {$acU \meq bcV$};
                    \node at (-3.0, -1.5) {$cU \meq baW_{2}$};
                    \draw[color=black] (-3.0,-1.2) -- (-1.5,-0.15000000000000002) node [midway, fill= white] {a};
                    \node at (0.0, -1.5) {$abW_{2} \meq cV$};
                    \draw[color=black] (0.0,-1.2) -- (-1.5,-0.15000000000000002) node [midway, fill= white] {b};
                    \node at (-6.0, -3.0) {$U \meq bcW_{4}$};
                    \draw[color=black] (-6.0,-2.7) -- (-3.0,-1.65) node [midway, fill= white] {c};
                    \node at (-3.0, -3.0) {$cbW_{4} \meq aW_{2}$};
                    \draw[color=black] (-3.0,-2.7) -- (-3.0,-1.65) node [midway, fill= white] {b};
                    \node at (0.0, -3.0) {$bW_{2} \meq caW_{3}$};
                    \draw[color=red] (0.0,-2.7) -- (0.0,-1.65) node [midway, fill= white] {a};
                    \node at (3.0, -3.0) {$acW_{3} \meq V$};
                    \draw[color=black] (3.0,-2.7) -- (0.0,-1.65) node [midway, fill= white] {c};
                    \node at (-4.5, -4.5) {$bW_{4} \meq acW_{6}$};
                    \draw[color=black] (-4.5,-4.2) -- (-3.0,-3.15) node [midway, fill= white] {c};
                    \node at (-1.5, -4.5) {$caW_{6} \meq W_{2} \meq cbW_{5}$};
                    \draw (-1.8,-4.2) -- (-3.0,-3.15) node [midway, fill= white] {a};
                    \draw (-1.2,-4.2) -- (0.0,-3.15) node [midway, fill= white] {b};
                    \node at (1.5, -4.5) {$bcW_{5} \meq aW_{3}$};
                    \draw[color=red] (1.5,-4.2) -- (0.0,-3.15) node [midway, fill= white] {c};
                    \node at (-7.5, -6.0) {$W_{4} \meq abW_{8}$};
                    \draw[color=black] (-7.5,-5.7) -- (-4.5,-4.65) node [midway, fill= white] {b};
                    \node at (-4.5, -6.0) {$baW_{8} \meq cW_{6}$};
                    \draw[color=black] (-4.5,-5.7) -- (-4.5,-4.65) node [midway, fill= white] {a};
                    \node at (-1.5, -6.0) {$aW_{6} \meq bW_{5}$};
                    \draw[color=black] (-1.5,-5.7) -- (-1.5,-4.65) node [midway, fill= white] {c};
                    \node at (1.5, -6.0) {$cW_{5} \meq abW_{7}$};
                    \draw[color=red] (1.5,-5.7) -- (1.5,-4.65) node [midway, fill= white] {b};
                    \node at (4.5, -6.0) {$baW_{7} \meq W_{3}$};
                    \draw[color=black] (4.5,-5.7) -- (1.5,-4.65) node [midway, fill= white] {a};
                \end{tikzpicture}
            }
            \caption{Graph obtained by applying the reduction rule to the initial assumption $acU\meq bcV$ in type $\widetilde{A_2}$. The red edges are the ones that witness the non-termination of the process.}
            \label{fig:non_termination_example}
            \end{figure}
        At this point, we know that the process will not end. Following the red edges, we can say that if there exist words $W_2$ and $V$ such that $abW_2 \meq cV$, then there exists a word $W_3$ such that $bW_2 \meq caW_3$, then some $W_5$ such that $bcW_5 \meq aW_3$, and finally some $W_7$ such that $abW_7 \meq cW_5$, ending with a relation of the same form as the one we started with. Thus, the process will never end. This path in the graph is a witness of the fact that the pair $(ac,bc)$ has no common right-multiple.
        \end{example}
    The question is now to know if such witness always exists when the construction of the graph does not end. To see this, we present the right-reversing algorithm of \cite{dehornoy_foundations_2015}, which is a more general process than the one we just presented.
        
\section{Right-reversing}
    \subsection{The right-reversing algorithm}
        We denote by $\fin{S}$ a formal copy of $S$ representing the inverses $\fin{s}$ of the elements $s$ of $S$. Words on $S \cup \fin{S}$ are called signed words on $S$. If $w$ is a signed word on $S$, then $\fin{w}$ is obtained from $w$ by reversing the order of the letters, and replacing each letter $s$ by $\fin{s}$ and each letter $\fin{s}$ by $s$. The word problem with respect to a group presentation $\pres{S}{R}$ is the problem of deciding if a signed word $w$ represent the identity element of the group. We define the notion of right-reversing, which is illustrated in figure \ref{fig:right_reversing}.
        \begin{definition}\label{defi:right_reverse}
            Let $w, w'$ be two signed words on $S$. We say that $w$ is \defit{right-reversible} to $w'$ in one step, written $w \rr^1 w'$, if there exists a subword $\fin{s}t$ of $w$ with $s,t$ in $S$, such that $w'$ is obtained from $w$ by replacing $\fin{s}t$ by $\rc{s}{t}\fin{\rc{t}{s}}$ (if $s=t$, then $\fin{s}t$ is simply removed). If $w$ has no subword of the form $\fin{s}t$, such that $\rc{s}{t}$ is defined, we say that $w$ is \defit{irreducible}. If there exists a sequence of words $w = w_0, w_1, \ldots, w_n = w'$ such that $w_i \rr^1 w_{i+1}$ for all $0\le i < n$, we say that $w$ is \defit{right-reversible} to $w'$ in $n$ steps and we write $w \rr^n w'$, and $w \rr^* w'$ (resp. $\rr^+w'$) if there exists $n \ge 0$ (resp. $n>0$) such that $w \rr^n w'$.
        \end{definition}
        \begin{figure}  
            \begin{center}  
                \begin{tikzpicture}[scale=2]
                    \node[scale = 0.5] (UL) at (0,0) {};
                    \node[scale = 0.5] (UR) at (1,0) {};
                    \node[scale = 0.5] (BL) at (0,-1) {};
                    \node[scale = 0.5] (BR) at (1,-1) {};
                    \node[scale = 0.5] (BBL) at (0,-2) {};
                    \node[scale = 0.5] (URR) at (2, 0) {};
                    \draw[-Stealth] (UL) -- (UR) node [midway, above] {$s$};
                    \draw[-Stealth] (UR) -- (BR) node [midway, right] {$\rc{s}{t}$};
                    \draw[-Stealth] (UL) -- (BL) node [midway, left] {$t$};
                    \draw[-Stealth] (BL) -- (BR) node [midway, below] {$\rc{t}{s}$};
                    \draw (BL) -- (BBL);
                    \draw (UR) -- (URR);
                    \draw[-Stealth, color = red] (-0.05,-1.95) -- (-0.05,0.05) -- (1.95,0.05) node [above] {$w$};
                    \draw[-Stealth, color = blue] (0.05,-1.95) -- (0.05, -1.05) -- (1.05,-1.05) -- (1.05,-0.05) -- (1.95,-0.05) node [below] {$w'$};
                \end{tikzpicture}
            \end{center}
            \caption{Illustration of the right-reversing process.}
            \label{fig:right_reversing}
        \end{figure}
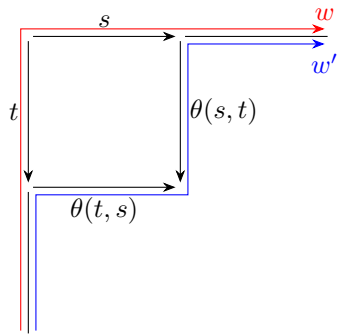
        We refer to \cite{dehornoy_foundations_2015} for a more detailed description of the properties of right-reversing. We present a quite large example to see how two algorithms behave very similarly.
        \begin{example}
            Let $M$ be the Artin-Tits monoid of type $B_3$. Computing the right-lcm of $aca$ and $bc$ with the two algorithms gives the commutative square common multiple $(aca)bcacb = (bc)acabca$ (see figure \ref{fig:comparison_algorithms}). Type $B_3$ has Coxeter diagram \begin{tikzcd}
                a && c && b
                \arrow[no head, from=1-3, to=1-5]
                \arrow["4", no head, from=1-1, to=1-3]
            \end{tikzcd}
        \end{example}
        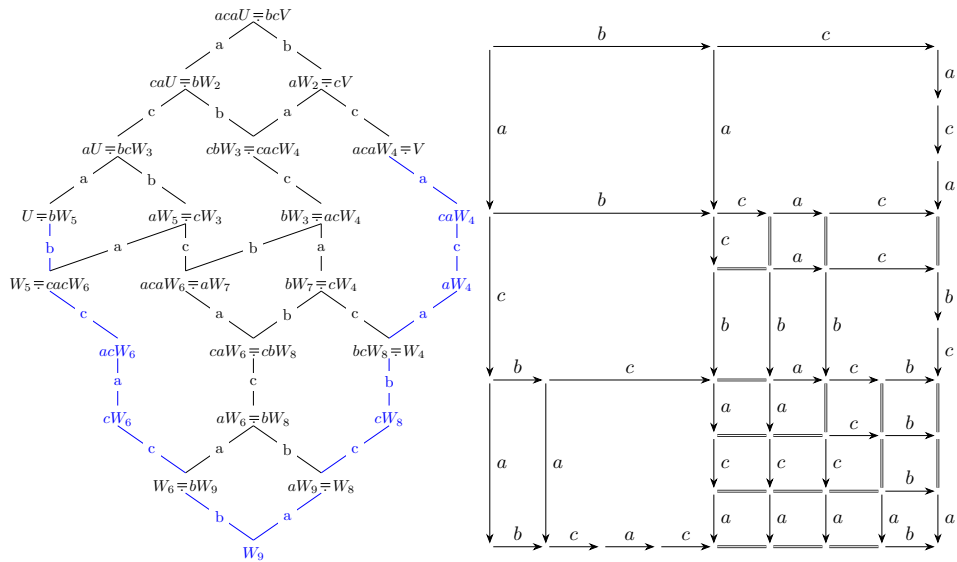
\begin{figure}
        \noindent
        \begin{minipage}[c]{0.5\linewidth}
            \resizebox{\linewidth}{!}{%
                \begin{tikzpicture}
                    \node[color=black] at (-1.5, -0.0) {$acaU \meq bcV$};
                    \node[color=black] at (-3.0, -1.5) {$caU \meq bW_{2}$};
                    \node[color=black] at (0.0, -1.5) {$aW_{2} \meq cV$};
                    \node[color=black] at (-4.5, -3.0) {$aU \meq bcW_{3}$};
                    \node[color=black] at (-1.5, -3.0) {$cbW_{3} \meq cacW_{4}$};
                    \node[color=black] at (1.5, -3.0) {$acaW_{4} \meq V$};
                    \node[color=black] at (-6.0, -4.5) {$U \meq bW_{5}$};
                    \node[color=black] at (-3.0, -4.5) {$aW_{5} \meq cW_{3}$};
                    \node[color=black] at (0.0, -4.5) {$bW_{3} \meq acW_{4}$};
                    \node[color=blue] at (3.0, -4.5) {$caW_{4}$};
                    \node[color=black] at (-6.0, -6.0) {$W_{5} \meq cacW_{6}$};
                    \node[color=black] at (-3.0, -6.0) {$acaW_{6} \meq aW_{7}$};
                    \node[color=black] at (0.0, -6.0) {$bW_{7} \meq cW_{4}$};
                    \node[color=blue] at (3.0, -6.0) {$aW_{4}$};
                    \node[color=blue] at (-4.5, -7.5) {$acW_{6}$};
                    \node[color=black] at (-1.5, -7.5) {$caW_{6} \meq cbW_{8}$};
                    \node[color=black] at (1.5, -7.5) {$bcW_{8} \meq W_{4}$};
                    \node[color=blue] at (-4.5, -9.0) {$cW_{6}$};
                    \node[color=black] at (-1.5, -9.0) {$aW_{6} \meq bW_{8}$};
                    \node[color=blue] at (1.5, -9.0) {$cW_{8}$};
                    \node[color=black] at (-3.0, -10.5) {$W_{6} \meq bW_{9}$};
                    \node[color=black] at (0.0, -10.5) {$aW_{9} \meq W_{8}$};
                    \node[color=blue] at (-1.5, -12.0) {$W_{9}$};
                    \draw[color=black] (-3.0,-1.2) -- (-1.5,-0.15000000000000002) node [midway, fill= white] {a};
                    \draw[color=black] (0.0,-1.2) -- (-1.5,-0.15000000000000002) node [midway, fill= white] {b};
                    \draw[color=black] (-4.5,-2.7) -- (-3.0,-1.65) node [midway, fill= white] {c};
                    \draw[color=black] (1.5,-2.7) -- (0.0,-1.65) node [midway, fill= white] {c};
                    \draw[color=black] (-1.5,-2.7) -- (-3.0,-1.65) node [midway, fill= white] {b};
                    \draw[color=black] (-1.5,-2.7) -- (0.0,-1.65) node [midway, fill= white] {a};
                    \draw[color=black] (-6.0,-4.2) -- (-4.5,-3.15) node [midway, fill= white] {a};
                    \draw[color=black] (-3.0,-4.2) -- (-4.5,-3.15) node [midway, fill= white] {b};
                    \draw[color=black] (0.0,-4.2) -- (-1.5,-3.15) node [midway, fill= white] {c};
                    \draw[color=black] (-6.0,-5.7) -- (-3.0,-4.65) node [midway, fill= white] {a};
                    \draw[color=black] (0.0,-5.7) -- (0.0,-4.65) node [midway, fill= white] {a};
                    \draw[color=black] (-3.0,-5.7) -- (-3.0,-4.65) node [midway, fill= white] {c};
                    \draw[color=black] (-3.0,-5.7) -- (0.0,-4.65) node [midway, fill= white] {b};
                    \draw[color=black] (1.5,-7.2) -- (0.0,-6.15) node [midway, fill= white] {c};
                    \draw[color=black] (-1.5,-7.2) -- (-3.0,-6.15) node [midway, fill= white] {a};
                    \draw[color=black] (-1.5,-7.2) -- (0.0,-6.15) node [midway, fill= white] {b};
                    \draw[color=black] (-1.5,-8.7) -- (-1.5,-7.65) node [midway, fill= white] {c};
                    \draw[color=black] (-3.0,-10.2) -- (-1.5,-9.15) node [midway, fill= white] {a};
                    \draw[color=black] (0.0,-10.2) -- (-1.5,-9.15) node [midway, fill= white] {b};
                    \draw[color=blue] (-6.0,-5.7) -- (-6.0,-4.65) node [midway, fill= white] {b};
                    \draw[color=blue] (-4.5,-7.2) -- (-6.0,-6.15) node [midway, fill= white] {c};
                    \draw[color=blue] (-4.5,-8.7) -- (-4.5,-7.65) node [midway, fill= white] {a};
                    \draw[color=blue] (-3.0,-10.2) -- (-4.5,-9.15) node [midway, fill= white] {c};
                    \draw[color=blue] (-1.5,-11.7) -- (-3.0,-10.65) node [midway, fill= white] {b};
                    \draw[color=blue] (3.0,-4.2) -- (1.5,-3.15) node [midway, fill= white] {a};
                    \draw[color=blue] (3.0,-5.7) -- (3.0,-4.65) node [midway, fill= white] {c};
                    \draw[color=blue] (1.5,-7.2) -- (3.0,-6.15) node [midway, fill= white] {a};
                    \draw[color=blue] (1.5,-8.7) -- (1.5,-7.65) node [midway, fill= white] {b};
                    \draw[color=blue] (0.0,-10.2) -- (1.5,-9.15) node [midway, fill= white] {c};
                    \draw[color=blue] (-1.5,-11.7) -- (0.0,-10.65) node [midway, fill= white] {a};
                \end{tikzpicture}
            }
        \end{minipage}%
        \begin{minipage}[c]{0.5\linewidth}
            \resizebox{\linewidth}{!}{%
                \begin{tikzpicture}
                    \node[scale = 0.5] (P0) at (0, -6) {};
                    \node[scale = 0.5] (P1) at (0, -3) {};
                    \node[scale = 0.5] (P2) at (0, 0) {};
                    \node[scale = 0.5] (P3) at (0, 3) {};
                    \node[scale = 0.5] (P4) at (4, 3) {};
                    \node[scale = 0.5] (P5) at (8, 3) {};
                    \node[scale = 0.5] (P6) at (4, 0) {};
                    \node[scale = 0.5] (P7) at (1, -3) {};
                    \node[scale = 0.5] (P8) at (4, -3) {};
                    \node[scale = 0.5] (P9) at (4, -1) {};
                    \node[scale = 0.5] (P10) at (1, -6) {};
                    \node[scale = 0.5] (P11) at (2, -6) {};
                    \node[scale = 0.5] (P12) at (3, -6) {};
                    \node[scale = 0.5] (P13) at (4, -6) {};
                    \node[scale = 0.5] (P14) at (4, -4) {};
                    \node[scale = 0.5] (P15) at (4, -5) {};
                    \node[scale = 0.5] (P16) at (5, 0) {};
                    \node[scale = 0.5] (P17) at (6, 0) {};
                    \node[scale = 0.5] (P18) at (8, 0) {};
                    \node[scale = 0.5] (P19) at (8, 2) {};
                    \node[scale = 0.5] (P20) at (8, 1) {};
                    \node[scale = 0.5] (P21) at (5, -1) {};
                    \node[scale = 0.5] (P22) at (5, -3) {};
                    \node[scale = 0.5] (P23) at (5, -4) {};
                    \node[scale = 0.5] (P24) at (5, -5) {};
                    \node[scale = 0.5] (P25) at (5, -6) {};
                    \node[scale = 0.5] (P26) at (6, -1) {};
                    \node[scale = 0.5] (P27) at (6, -3) {};
                    \node[scale = 0.5] (P28) at (6, -4) {};
                    \node[scale = 0.5] (P29) at (6, -5) {};
                    \node[scale = 0.5] (P30) at (6, -6) {};
                    \node[scale = 0.5] (P31) at (8, -1) {};
                    \node[scale = 0.5] (P32) at (7, -3) {};
                    \node[scale = 0.5] (P33) at (8, -3) {};
                    \node[scale = 0.5] (P34) at (8, -2) {};
                    \node[scale = 0.5] (P35) at (7, -4) {};
                    \node[scale = 0.5] (P36) at (7, -5) {};
                    \node[scale = 0.5] (P37) at (7, -6) {};
                    \node[scale = 0.5] (P38) at (8, -4) {};
                    \node[scale = 0.5] (P39) at (8, -5) {};
                    \node[scale = 0.5] (P40) at (8, -6) {};
                    \draw[-Stealth] (P1) edge node[auto]{$a$} (P0);
                    \draw[-Stealth] (P2) edge node[auto]{$c$} (P1);
                    \draw[-Stealth] (P3) edge node[auto]{$a$} (P2);
                    \draw[-Stealth] (P3) edge node[auto]{$b$} (P4);
                    \draw[-Stealth] (P4) edge node[auto]{$c$} (P5);
                    \draw[-Stealth] (P2) edge node[auto]{$b$} (P6);
                    \draw[-Stealth] (P4) edge node[auto]{$a$} (P6);
                    \draw[-Stealth] (P1) edge node[auto]{$b$} (P7);
                    \draw[-Stealth] (P7) edge node[auto]{$c$} (P8);
                    \draw[-Stealth] (P6) edge node[auto]{$c$} (P9);
                    \draw[-Stealth] (P9) edge node[auto]{$b$} (P8);
                    \draw[-Stealth] (P0) edge node[auto]{$b$} (P10);
                    \draw[-Stealth] (P7) edge node[auto]{$a$} (P10);
                    \draw[-Stealth] (P10) edge node[auto]{$c$} (P11);
                    \draw[-Stealth] (P11) edge node[auto]{$a$} (P12);
                    \draw[-Stealth] (P12) edge node[auto]{$c$} (P13);
                    \draw[-Stealth] (P8) edge node[auto]{$a$} (P14);
                    \draw[-Stealth] (P14) edge node[auto]{$c$} (P15);
                    \draw[-Stealth] (P15) edge node[auto]{$a$} (P13);
                    \draw[-Stealth] (P6) edge node[auto]{$c$} (P16);
                    \draw[-Stealth] (P16) edge node[auto]{$a$} (P17);
                    \draw[-Stealth] (P17) edge node[auto]{$c$} (P18);
                    \draw[-Stealth] (P5) edge node[auto]{$a$} (P19);
                    \draw[-Stealth] (P19) edge node[auto]{$c$} (P20);
                    \draw[-Stealth] (P20) edge node[auto]{$a$} (P18);
                    \draw[] (P9) edge[double] node[auto]{} (P21);
                    \draw[] (P16) edge[double] node[auto]{} (P21);
                    \draw[] (P8) edge[double] node[auto]{} (P22);
                    \draw[-Stealth] (P21) edge node[auto]{$b$} (P22);
                    \draw[] (P14) edge[double] node[auto]{} (P23);
                    \draw[-Stealth] (P22) edge node[auto]{$a$} (P23);
                    \draw[] (P15) edge[double] node[auto]{} (P24);
                    \draw[-Stealth] (P23) edge node[auto]{$c$} (P24);
                    \draw[] (P13) edge[double] node[auto]{} (P25);
                    \draw[-Stealth] (P24) edge node[auto]{$a$} (P25);
                    \draw[-Stealth] (P21) edge node[auto]{$a$} (P26);
                    \draw[] (P17) edge[double] node[auto]{} (P26);
                    \draw[-Stealth] (P22) edge node[auto]{$a$} (P27);
                    \draw[-Stealth] (P26) edge node[auto]{$b$} (P27);
                    \draw[] (P23) edge[double] node[auto]{} (P28);
                    \draw[] (P27) edge[double] node[auto]{} (P28);
                    \draw[] (P24) edge[double] node[auto]{} (P29);
                    \draw[-Stealth] (P28) edge node[auto]{$c$} (P29);
                    \draw[] (P25) edge[double] node[auto]{} (P30);
                    \draw[-Stealth] (P29) edge node[auto]{$a$} (P30);
                    \draw[-Stealth] (P26) edge node[auto]{$c$} (P31);
                    \draw[] (P18) edge[double] node[auto]{} (P31);
                    \draw[-Stealth] (P27) edge node[auto]{$c$} (P32);
                    \draw[-Stealth] (P32) edge node[auto]{$b$} (P33);
                    \draw[-Stealth] (P31) edge node[auto]{$b$} (P34);
                    \draw[-Stealth] (P34) edge node[auto]{$c$} (P33);
                    \draw[-Stealth] (P28) edge node[auto]{$c$} (P35);
                    \draw[] (P32) edge[double] node[auto]{} (P35);
                    \draw[] (P29) edge[double] node[auto]{} (P36);
                    \draw[] (P35) edge[double] node[auto]{} (P36);
                    \draw[] (P30) edge[double] node[auto]{} (P37);
                    \draw[-Stealth] (P36) edge node[auto]{$a$} (P37);
                    \draw[-Stealth] (P35) edge node[auto]{$b$} (P38);
                    \draw[] (P33) edge[double] node[auto]{} (P38);
                    \draw[-Stealth] (P36) edge node[auto]{$b$} (P39);
                    \draw[] (P38) edge[double] node[auto]{} (P39);
                    \draw[-Stealth] (P37) edge node[auto]{$b$} (P40);
                    \draw[-Stealth] (P39) edge node[auto]{$a$} (P40);
                \end{tikzpicture}\vspace{-10ex}
        }
        \end{minipage}
        \caption{On the left, applying the reduction rule, and on the right, applying right-reversing, to compute the right-lcm of $aca$ and $bc$ in type $B_3$.}
        \label{fig:comparison_algorithms}
        \end{figure}

        One fundamental property is the uniqueness of the maximal right-reversing diagram, which is a consequence of the following confluence property of right-reversing.
        \begin{lemma}\label{lemm:rr_confluence}
            Let $w, w_1, w_2$ be three signed words on $S$. We assume $w \rr^{n_1} w_1$ and $w \rr^{n_2} w_2$. Then there exists a signed word $w'$ such that $w_1 \rr^{m_1} w'$ and $w_2 \rr^{m_2} w'$,  and we have $m_1\le n_2, m_2\le n_1$ and $n_1 + m_1 = n_2 + m_2$.\\
            In particular, if $w\rr^*w_1$ and $w\rr^*w_2$, such that $w_1$ and $w_2$ are irreducible, then $w_1 = w_2$.
        \end{lemma}
        \begin{proof}
            The proof is an induction on $n_1+n_2$. Note that if $n_1$ (resp. $n_2$) is $0$, then $w'=w_2$ (resp. $w'=w_1$) and $m_1 = n_2$ (resp. $m_2 = n_1$) works. We start with the case $n_1 = n_2 = 1$. If the two reductions take place in the same place in $w$, then $w_1$ and $w_2$ are equal, and $w'=w_1$ works with $m_1=m_2=0$. If the reduction take place in different places, then the two corresponding reversed subwords $\fin{s}t$ are disjoint. So we can apply one after the other, giving $w\rr^2w'$ and $m_1=m_2=1$. Assume now that $n_1+n_2\ge 3$ and both $n_1 > 0$ and $n_2 > 0$. Without loss of generality, we can assume $n_1 \ge 2$. Then, there exists a signed word $w_1'$ such that $w\rr^1 w_1' \rr^{n_1-1}w_1$. Then the induction hypothesis applies to $w, w_1'$ and $w_2$ since $1+n_2 < n_1 + n_2$. So there exists a signed word $w_2'$ such that $w_1'\rr^{p_1} w_2'$ and $w_2 \rr^{p_2} w_2'$, with $p_1\le n_2, p_2\le 1$ and $1 + p_1 = n_2 + p_2$. Now, the induction hypothesis applies to $w_1', w_1, w_2'$ because $n_1-1+p_1 \le n_1-1+n_2<n_1+n_2$. So there exists a signed word $w'$ such that $w_1\rr^{q_1} w'$ and $w_2'\rr^{q_2} w'$, with $q_1\le p_1$ and $q_2\le n_1-1$ and $n_1-1 + q_1 = p_1 + q_2$. Finally, we have $w_1 \rr^{m_1} w'$ and $w_2 \rr^{m_2} w'$, with $m_1 = q_1\le p_1 \le n_2$ and $m_2 = q_2 + p_2 \le n_1-1 + 1 = n_1$, and $n_1 + m_1 = 1+n_1-1+q_1=1+ p_1 + q_2 = n_2 + p_2 + q_2 = n_2 + m_2$.
        \end{proof}
        When right-reversing ends successfully, that is, when $\fin{u}v\rr^* \fin{v'}u'$, then the equality $uv' \meq vu'$ holds in the presented monoid. But the converse is not always true. This is the object of the notion of completeness, that must be understood in parallel with the reduction rule of Definition (\ref{defi:reduction_rule}), and is illustrated in figure \ref{fig:completeness}.\\
        \begin{definition}{\label{defi:completeness}}
            Let $\pres{S}{R}^+$ be a right-complemented presentation. Let $(u,v,\hat{u},\hat{v})$ such that $u\hat{v} \meq v\hat{u}$. We say that the quadruple $(u,v,\hat{u},\hat{v})$ is \defit{factorable through right-reversing} or $\rr$\defit{-factorable} if there exist three words $u',v',w'$ such that $\fin{u}v\rr^* \fin{u'}v'$ and $\fin{u'}w' \meq \hat{u}$ and $v'w' \meq \hat{v}$.\\
            We say that right-reversing is \defit{complete} for $\pres{S}{R}^+$ if any quadruple $(u,v,\hat{u},\hat{v})$ such that $u\hat{v} \meq v\hat{u}$ is $\rr$-factorable.
        \end{definition}
        \begin{figure}[b]
            \begin{center}
                \begin{tikzpicture}[scale = 0.8]\label{figimported:rr_factorable}
                    \node[scale = 0.5] (P0) at (1, -1) {};
                    \node[scale = 0.5] (P1) at (4, -1) {};
                    \node[scale = 0.5] (P2) at (6, -3) {};
                    \node[scale = 0.5] (P3) at (1, -4) {};
                    \node[scale = 0.5] (P4) at (4, -4) {};
                    \node[scale = 0.5] (P5) at (3, -6) {};
                    \node[scale = 0.5] (P6) at (6, -6) {};
                    \draw[-Stealth] (P0) edge [] node[auto]{$v$} (P1);
                    \draw[-Stealth] (P0) edge [swap] node[auto]{$u$} (P3);
                    \draw[-Stealth] (P0) edge [draw=none] node[midway, fill=white]{${\rr^*}$} (P4);
                    \draw[-Stealth] (P1) edge [color={rgb,255:red,92;green,92;blue,214}, dashed] node[auto]{${u'}$} (P4);
                    \draw[-Stealth] (P1) edge [bend right=-18] node[auto]{${\hat{u}}$} (P6);
                    \draw[-Stealth] (P2) edge [draw=none] node[midway, fill=white]{$\meq$} (P4);
                    \draw[-Stealth] (P3) edge [swap, color={rgb,255:red,92;green,92;blue,214}, dashed] node[auto]{${v'}$} (P4);
                    \draw[-Stealth] (P3) edge [swap, bend right=18] node[auto]{${\hat{v}}$} (P6);
                    \draw[-Stealth] (P4) edge [color={rgb,255:red,92;green,92;blue,214}, dashed] node[auto]{${w'}$} (P6);
                    \draw[-Stealth] (P5) edge [draw=none] node[midway, fill=white]{$\meq$} (P4);
                \end{tikzpicture}
            \end{center}
            \caption{Illustration of the notion of quadruple factorable through right-reversing.}
            \label{fig:completeness}
        \end{figure}
        The equivalence between the notion of completeness and the reduction rule is stated by the following proposition:
        \begin{proposition}\label{prop:comp_redrule}
            Let $\pres{S}{R}^+$ be a right-complemented presentation. Then right-reversing is complete if and only if the reduction rule holds.
        \end{proposition}
        \begin{proof}
            First, assume right-reversing is complete. Let $X,Y \in S^*$ and $a,b$ be two letters in $S$ such that $aX \meq bY$. Then, by completeness, the  $(a,b,Y,X)$ is $\rr$-factorable. So there exist $a', b'$, and $W$ such that $\fin{a}b \rr^* b'\fin{a'}$ and $X \meq a'W$ and $Y \meq b'W$. But by uniqueness of the right-reversing diagram, we necessarily have $a' = \rc{a}{b}$ and $b' = \rc{b}{a}$, and the reduction rule holds.\\
            The converse is essentially the of proof of point \eqref{reduc:conlcm} of Lemma \ref{lemm:reduction} . Indeed, the common multiple built in the proof is obtained by right-reversing.
        \end{proof}
        The completeness of right-reversing allows to compute the right-lcm of two words $u$ and $v$.
        \begin{proposition}\label{prop:complete_compute_lcm}
            Let $\pres{S}{R}^+$ be a right-complemented presentation such that right-reversing is complete. Then for all $u,v \in S^*$, $u$ and $v$ have a right-lcm if and only if there exist $u',v'$ such that $\fin{u}v \rr^* v'\fin{u'}$, and in this case, $uv' \meq vu'$ is the right-lcm of $u$ and $v$.
        \end{proposition}
        \begin{proof}\label{prop:complete_lcm}
            Assume $u$ and $v$ have a right-lcm $u\hat{v}=v\hat{u}$. Then, by completeness, the quadruple $(u,v,\hat{u},\hat{v})$ is $\rr$-factorable. So there exist $u',v',w'$ such that $\fin{u}v \rr^* v'\fin{u'}$ and $v'w' \meq \hat{v}$ so $uv'w' \meq u\hat{v}$, i.e. $uv'$ is a left-divisor of $u\hat{v}$. By minimality of the lcm, we have $uv' \meq u\hat{v}$ is the right-lcm of $u$ and $v$.\\
            Conversely, assume that there exist $u',v'$ such that $\fin{u}v \rr^* v'\fin{u'}$. Then for every common right-multiple of $u$ and $v$, written $u\hat{v} = v\hat{u}$, the quadruple $(u,v,\hat{u},\hat{v})$ is $\rr$-factorable. Then, there exists $w'$ such that $uv'w' \meq u\hat{v}$ so $uv'$ left-divides $u\hat{v}$ and it is a right-lcm of $u$ and $v$.   
        \end{proof}
        To conclude this quick presentation of the right-reversing algorithm, we need some criterion to know if right-reversing is complete for some right-complemented presentation $\pres{S}{R}^+$, one that works on Artin-Tits monoids. This criterion will use an extension of the right-complement $\rcb$.
        \begin{definition}
            Let $\pres{S}{R}^+$ be a right-complemented presentation with associated right-complement $\rcb$. Let $u,v \in S^*$. If there exist $u',v' \in S^*$ such that $\fin{u}v \rr^* v'\fin{u'}$, then such $u'$ and $v'$ are unique because of Lemma \ref{lemm:rr_confluence}. In this case, we set $\erc{u}{v}=u'$ and $\erc{v}{u}=v'$. $\ercb$ is called the \defit{extended syntactic right-complement} associated with the right-complemented presentation $\pres{S}{R}^+.$ It is a partial map from $S^*\times S^*$ to $S^*$\\
            Let $u,v,w \in S$ such that $\erc{u}{v}, \erc{u}{w}$ are defined, and $\erc{\erc{u}{v}}{\erc{u}{w}}$ is defined. Then we set $\erct{u}{v}{w} = \erc{\erc{u}{v}}{\erc{u}{w}}$.
        \end{definition}
        \begin{remark}
            We can swap the entries for the syntactic right-complement $\ercb$ and it is still defined. However, we can't do any permutation of $u,v,w$ for the syntactic right-complement $\erctb$. We can swap the last two entries, that is, if $\erct{u}{v}{w}$ is defined, then so is $\erct{u}{w}{v}$ (but they are not equal). The swapping of the first two entries is the object of the criterion that we present now.
        \end{remark}
        \begin{definition}
            Let $\pres{S}{R}^+$ be a right-complemented presentation with associated right-complement $\rcb$. Let $u,v,w \in S^*$. We say that the $\rcb$-cube condition holds for $u,v,w$ if either both $\erct{u}{v}{w}$ and $\erct{v}{u}{w}$ are defined and $\meq$-equivalent, or neither is defined. We say that the $\rcb$-cube condition holds on $S$ if it holds for all $u,v,w \in S$.
        \end{definition}
        The name "cube" refers to the characterization of the $\rcb$-cube condition illustrated in Figure \ref{fig:cube_condition}, called simply the cube condition (see Lemma II.4.55 in \cite{dehornoy_foundations_2015} for the proof)
        \begin{lemma} The $\rcb$-cube condition holds on $S$ if and only if any diagram on the left of Figure \ref{fig:cube_condition}, with $u,v,w \in S$, can be completed into the cubic diagram on the right.
        \end{lemma}
        \begin{figure}[h!]
            \resizebox{\linewidth}{!}{%
                \begin{tikzpicture}\label{figimported:cube_condition1}
                    \node[scale = 0.5] (P0) at (3, -1) {};
                    \node[scale = 0.5] (P1) at (6, -1) {};
                    \node[scale = 0.5] (P2) at (10, -1) {};
                    \node[scale = 0.5] (P3) at (13, -1) {};
                    \node[scale = 0.5] (P4) at (1, -2) {};
                    \node[scale = 0.5] (P5) at (8, -2) {};
                    \node[scale = 0.5] (P6) at (11, -2) {};
                    \node[scale = 0.5] (P7) at (3, -4) {};
                    \node[scale = 0.5] (P8) at (6, -4) {};
                    \node[scale = 0.5] (P9) at (10, -4) {};
                    \node[scale = 0.5] (P10) at (13, -4) {};
                    \node[scale = 0.5] (P11) at (1, -5) {};
                    \node[scale = 0.5] (P12) at (4, -5) {};
                    \node[scale = 0.5] (P13) at (8, -5) {};
                    \node[scale = 0.5] (P14) at (11, -5) {};
                    \draw[-Stealth] (P0) edge [] node[auto]{$v$} (P1);
                    \draw[-Stealth] (P0) edge [swap] node[auto]{$u$} (P4);
                    \draw[-Stealth] (P0) edge [] node[auto]{$w$} (P7);
                    \draw[-Stealth] (P0) edge [sloped, draw=none] node[auto, allow upside down]{${\rr^*}$} (P8);
                    \draw[-Stealth] (P0) edge [sloped, draw=none] node[auto, allow upside down]{${\rr^*}$} (P11);
                    \draw[-Stealth] (P1) edge [swap, color={rgb,255:red,92;green,92;blue,214}] node[auto, pos=0.2]{${\erc{v}{w}}$} (P8);
                    \draw[-Stealth] (P2) edge [] node[auto]{$v$} (P3);
                    \draw[-Stealth] (P2) edge [swap] node[auto]{$u$} (P5);
                    \draw[-Stealth] (P2) edge [sloped, draw=none] node[auto, allow upside down, pos=0.6]{${\rr^*}$} (P6);
                    \draw[-Stealth] (P2) edge [color={rgb,255:red,204;green,204;blue,204}] node[auto]{} (P9);
                    \draw[-Stealth] (P3) edge [color={rgb,255:red,92;green,92;blue,214}, dashed] node[auto, pos=0.7]{${\erc{v}{u}}$} (P6);
                    \draw[-Stealth] (P3) edge [color={rgb,255:red,92;green,92;blue,214}] node[auto]{${\erc{v}{w}}$} (P10);
                    \draw[-Stealth] (P3) edge [sloped, draw=none] node[auto, allow upside down, pos=0.4]{${\rr^*}$} (P14);
                    \draw[-Stealth] (P4) edge [color={rgb,255:red,92;green,92;blue,214}] node[auto, pos=0.1]{${\erc{u}{w}}$} (P11);
                    \draw[-Stealth] (P5) edge [swap, color={rgb,255:red,92;green,92;blue,214}, dashed] node[auto]{${\erc{u}{v}}$} (P6);
                    \draw[-Stealth] (P5) edge [color={rgb,255:red,92;green,92;blue,214}] node[auto, pos=0.8]{${\erc{u}{w}}$} (P13);
                    \draw[-Stealth] (P5) edge [sloped, draw=none] node[auto, allow upside down, pos=0.3]{${\rr^*}$} (P14);
                    \draw[-Stealth] (P6) edge [color={rgb,255:red,214;green,92;blue,92}, dashed] node[midway, fill=white]{${\erct{u}{v}{w}\meq\erct{v}{u}{w}}$} (P14);
                    \draw[-Stealth] (P7) edge [color={rgb,255:red,92;green,92;blue,214}] node[auto]{${\erc{w}{v}}$} (P8);
                    \draw[-Stealth] (P7) edge [color={rgb,255:red,92;green,92;blue,214}] node[auto]{${\erc{w}{u}}$} (P11);
                    \draw[-Stealth] (P7) edge [sloped, draw=none] node[auto, allow upside down, pos=0.7]{${\rr^*}$} (P12);
                    \draw[-Stealth] (P8) edge [color={rgb,255:red,214;green,92;blue,92}] node[auto, pos=0.7]{${\erct{w}{v}{u}}$} (P12);
                    \draw[-Stealth] (P9) edge [color={rgb,255:red,204;green,204;blue,204}] node[auto]{} (P10);
                    \draw[-Stealth] (P9) edge [color={rgb,255:red,204;green,204;blue,204}] node[auto]{} (P13);
                    \draw[-Stealth] (P10) edge [color={rgb,255:red,214;green,92;blue,92}] node[auto]{${\erct{w}{v}{u}}$} (P14);
                    \draw[-Stealth] (P11) edge [swap, color={rgb,255:red,214;green,92;blue,92}] node[auto]{${\erct{w}{u}{v}}$} (P12);
                    \draw[-Stealth] (P13) edge [swap, color={rgb,255:red,214;green,92;blue,92}] node[auto]{${\erct{w}{u}{v}}$} (P14);
                \end{tikzpicture}
                }
            \caption{The $\rcb$-cube condition.}
            \label{fig:cube_condition}
        \end{figure}
        When trying to prove the reduction rule, (or the completeness of right-reversing), the proof works on a nested induction, first on the length of the common multiple $aX=bY$, and then on the length of a shortest derivation from $aX$ to $bY$. In the induction step appears the relation $aX \meq cZ \meq bY$, and applying the reduction rule to $aX \meq cZ$ and $cZ \meq bY$ gives two expressions for $Z$, $\rc{c}{a}W_1\meq\rc{c}{b}W_2$. Then, the induction hypothesis allows to use the reduction rule on the equation $\rc{c}{a}W_1\meq\rc{c}{b}W_2$, or in other words, to compute $\erc{\rc{c}{a}}{\rc{c}{b}}=\erct{c}{b}{a}$. To complete the induction step, we need to show that $\rc{a}{b}$ left-divides $X$ and that $\rc{b}{a}$ left-divides $Y$. Then, the $\rcb$-cube appears naturally, and allows to complete the induction step. See in \cite{dehornoy_foundations_2015} the Proposition 4.62 and the proof in the Appendix for details.\\
        \begin{proposition}\label{prop:cube_complete}
            Let $\pres{S}{R}^+$ be a right-complemented presentation that is right-noetherian, and the $\rcb$-cube condition holds on $S$, then right-reversing is complete for $\pres{S}{R}^+$.
        \end{proposition}
        The next step towards computability of right-lcms is to discuss termination issues.
    \subsection{Termination}
        We have seen in Remark \ref{remark:termination} that successive applications of the reduction rule, and equivalently the right-reversing algorithm, can sometimes continue indefinitely. In the Example \ref{ex:termination}, we have seen that the process can be stopped because of the existence of a witness that the process will never end. 
        Here is the example presented with the right-reversing algorithm:
        \begin{example}\label{ex:rrtermination}
            The right-reversing algorithm applied to the word $\fin{bc}a$ with the right-complemented presentation of $\widetilde{A_2}$ never ends because it reappear in its own right-reversing process. In figure \ref{fig:rr_witness}, we see that when starting the right-reversing process with the signed word $\fin{bc}a$, we get the a signed word conating the factor $\fin{a}bc$ after three steps. If we keep going, after three more steps, we will get $\fin{bc}a$ as a subword. This shows that the process will never end, and we can stop the process with the result that $bc$ and $a$ have no common right-multiple in the Artin-Tits monoid of type $\widetilde(A_2)$.
        \end{example}
        \begin{figure}[h!]
            \begin{center}
                \begin{tikzpicture}[scale = 0.8]\label{figimported:rr_witness}
                    \node[scale = 0.5] (P0) at (1, -1) {};
                    \node[scale = 0.5] (P1) at (5, -1) {};
                    \node[scale = 0.5] (P2) at (5, -2) {};
                    \node[scale = 0.5] (P3) at (1, -3) {};
                    \node[scale = 0.5] (P4) at (3, -3) {};
                    \node[scale = 0.5] (P5) at (5, -3) {};
                    \node[scale = 0.5] (P6) at (5, -4) {};
                    \node[scale = 0.5] (P7) at (3, -5) {};
                    \node[scale = 0.5] (P8) at (4, -5) {};
                    \node[scale = 0.5] (P9) at (5, -5) {};
                    \node[scale = 0.5] (P10) at (1, -6) {};
                    \node[scale = 0.5] (P11) at (2, -6) {};
                    \node[scale = 0.5] (P12) at (3, -6) {};
                    \draw[-Stealth] (P0) edge [color={rgb,255:red,214;green,92;blue,92}] node[auto]{$a$} (P1);
                    \draw[-Stealth] (P0) edge [color={rgb,255:red,214;green,92;blue,92}] node[auto]{$b$} (P3);
                    \draw[-Stealth] (P1) edge [] node[auto]{$b$} (P2);
                    \draw[-Stealth] (P2) edge [] node[auto]{$a$} (P5);
                    \draw[-Stealth] (P3) edge [] node[auto]{$a$} (P4);
                    \draw[-Stealth] (P3) edge [color={rgb,255:red,214;green,92;blue,92}] node[auto]{$c$} (P10);
                    \draw[-Stealth] (P4) edge [] node[auto]{$b$} (P5);
                    \draw[-Stealth] (P4) edge [] node[auto]{$c$} (P7);
                    \draw[-Stealth] (P5) edge [] node[auto]{$c$} (P6);
                    \draw[-Stealth] (P6) edge [] node[auto]{$b$} (P9);
                    \draw[-Stealth] (P7) edge [color={rgb,255:red,92;green,92;blue,214}] node[auto]{$b$} (P8);
                    \draw[-Stealth] (P7) edge [color={rgb,255:red,92;green,92;blue,214}] node[auto]{$a$} (P12);
                    \draw[-Stealth] (P8) edge [color={rgb,255:red,92;green,92;blue,214}] node[auto]{$c$} (P9);
                    \draw[-Stealth] (P10) edge [] node[auto]{$a$} (P11);
                    \draw[-Stealth] (P11) edge [] node[auto]{$c$} (P12);
                \end{tikzpicture}
            \end{center}
            \caption{The right-reversing diagram witnessing the non-termination of the right-reversing process starting with $\fin{bc}a$.}
            \label{fig:rr_witness}
        \end{figure}

        We want to formalize this notion of cyclic witness. We start with some notation and vocabulary.\\
        \begin{definition}
            Let $u,v, u', v'$ be words of $S^*$. We say that the pair $(u',v')$ is a \defit{reduction factor} (resp. \defit{strict reduction factor})  of the pair $(u,v)$, written $(u,v) \rf (u',v')$ (resp. $(u,v) \strrf (u',v')$) if there exist some signed words $w_1, w_2$ such that $\fin{u}v \rr^n w_1\fin{u'}v'w_n$ for some $n \ge 0$ (resp. $n > 0$).\\
            We say that a pair $(u,v)$ is \defit{periodic} if $(u,v) \strrf (u,v)$, and we say that it is \defit{eventually periodic} if there exists some $(u',v')$ periodic such that $(u,v) \rf (u',v')$.
            We say that a pair $(u,v)$ is \defit{failing} if there exist some letters of $S$, $(s,t)$ such that $(u,v) \rf (s,t)$ and $\rc{s}{t}$ is undefined.
        \end{definition}
        Failure and (eventual-)periodicity can effectively be detected in an algorithm. Starting with a failing pair $(u,v)$ and applying the right-reversing to the signed word $\fin{u}v$, a factor $\fin{s}t$ with $\rc{s}{t}$ undefined will appear Starting with a periodic pair $(u,v)$, we can read at each steps of the right-reversing all the subwords $\fin{u_1}v_1$ appearing and the word $\fin{u}v$ will appear eventually. Starting with an eventually periodic pair $(u,v)$, one can launch in parallel the algorithm to check periodicity on all subword $\fin{u_1}v_1$ appearing in the right-reversing of $(u,v)$ and one will eventually end. More efficient ways of detecting (eventual-)periodicity can be implemented based on the induction in Lemma \ref{lemm:trichotomy_gar} below (see the repository \url{https://github.com/domyrmininon/compute_gf}). Then, in order to have an lcm algorithm that always ends, we have to prove some kind of trichotomy between failing, eventual periodicity, and success of the algorithm. We need here an additional assumption, the existence of a finite Garside family. See \cite{dehornoy_foundations_2015} III for the definition of a Garside family in general. In the case of a right-complemented presentation with complete right-reversing, a family $G$ of elements of $M$ is a Garside family if and only if it generates $M$, and it is closed under right-lcm and right-complement (or equivalently, right-lcm and right-divisor) (see \cite{dehornoy_foundations_2015}, Proposition IV.2.40). The existence of a finite Garside family $G$ in this framework allows to give a new right-complemented presentation of the same monoid with generators $G$, and short relations, which largely simplifies the problem at the cost of computing this Garside family. As explained in the introduction, our work provides a way to avoid the computation of the whole Garside family. The main property of having a Garside family is presented in the following lemma.
        \begin{lemma}
            Let $M$ be a gcd monoid and $G\subset M$ be a family closed under right-lcm and right-divisor. Denote by $\wl{u}_G$ the minimal length of a product of elements of $G$ that is equal to $u$. Then, for all $u,v \in M$, if $u$ and $v$ have a right-lcm $uv'=vu'$ with $u',v' \in M$, then $\wl{u'}_G \le \wl{u}_G$ and $\wl{v'}_G \le \wl{v}_G$.
        \end{lemma}
        \begin{proof}
            Let $u,v \in M$ that have a right-lcm $uv'=vu'$. The proof is by induction on $\wl{u}_G + \wl{v}_G$. If $u$ (resp. $v$) is $1$, then $u'=u$ (resp. $v'=v$) and the lemma holds. If $\wl{u}_G = \wl{v}_G = 1$, then $u,v \in G$ and if $uv'=vu'$ is the right-lcm of $u$ and $v$, then it is in $G$ since $G$ is closed under right-lcm, and then $u'$ and $v'$ are in $G$ since $G$ is closed under right-divisor, so $\wl{u'}_G \le 1=$ and $\wl{v'}_G \le 1$. Assume now that $\wl{u}_G + \wl{v}_G \ge 3$ and both $u$ and $v$ are not trivial. Write $u = u_1u_2$ with $u_1\in G$ and $\wl{u_2}_G = \wl{u}_G - 1$. By Lemma \ref{lemm:concat_lcm}, $v$ and $u_1$ have a right-lcm $vu_1' = u_1v_1'$, $v_1'$ and $u_2$ have a right-lcm $v_1'u_2' = u_2v'$ and $u_1'u_2'=u'$. Then, the induction hypothesis applies to $(u_1,v)$ and gives $\wl{u_1'}_G \le \wl{u_1}_G\le 1$ and $\wl{v_1'}_G \le \wl{v}_G$. Then $\wl{v_1'}_G+\wl{u_2}_G\le \wl{v}_G + \wl{u}_G - 1$, and the induction hypothesis applies to $(v_1',u_2)$ and gives $\wl{u_2'}_G \le \wl{u_2}_G$ and $\wl{v'}_G \le \wl{v_1'}_G$. Finally, we have $\wl{u'}_G = \wl{u_1'}_G + \wl{u_2'}_G \le \wl{u_1}_G + \wl{u_2}_G = \wl{u}_G$ and $\wl{v'}_G = \wl{v_1'}_G \le \wl{v}_G$.
        \end{proof}
        \begin{theorem}\label{theo:trichotomy}
            Let $\pres{S}{R}^+$ be a right-complemented presentation of a monoid $M$ for which right-reversing is complete. We assume that $M$ has a finite Garside family $G$. Let $u, v$ be two words of $S^*$. Then one of the following holds:
            \begin{enumerate}
                \item\label{theo:trichotomy1} $(u,v)$ is failing
                \item\label{theo:trichotomy2} $(u,v)$ is eventually periodic
                \item\label{theo:trichotomy3} $\erc{u}{v}$ is defined, i.e. there exist some positive words $u',v'$ on $S$ such that $\fin{u}v \rr^* v'\fin{u'}$.
            \end{enumerate}
        \end{theorem}
        \begin{remark}
            The first two points of the theorem are not exclusive. The last point corresponds to the case where $u$ and $v$ have a right-lcm (Proposition \ref{prop:complete_lcm}). It is exclusive with the first two because of the confluence of right-reversing. For point \eqref{theo:trichotomy1}, any signed word $w'$ derived from a signed word $w'$ that contains a subword $\fin{s}t$ such that $\rc{s}{t}$ is undefined still contains the same subword, and can never be of the form $u'\fin{v'}$. For point \eqref{theo:trichotomy2}, if $(u,v)$ is periodic, there is an arbitrarily large right-reversing sequence starting from $\fin{u}v$ which is incompatible with a finite right-reversing sequence $\fin{u}v \rr^* v'\fin{u'}$.
        \end{remark}
        Before proving this proposition we start with the following trivial properties:
        \begin{lemma}\label{lemm:trichotomy_rf}
            Let $u,v,u',v' \in S^*$. If $(u',v')$ is failing, and $(u,v) \rf (u',v')$, then $(u,v)$ is failing.\\
            If $(u',v')$ is eventually periodic, and $(u,v) \rf (u',v')$, then $(u,v)$ is eventually periodic.\\
        \end{lemma}
        The proof is then by induction on the $G-$length of the pair $(u,v)$, and reduces to the case where $u$ and $v$ are in $G$.
        \begin{lemma}\label{lemm:trichotomy_ind}
            Let $\pres{S}{R}^+$ be a right-complemented presentation of a monoid $M$ for which right-reversing is complete. Let $G$ be a Garside family of $M$. We assume that for all $u,v \in G$, $u,v$ is either failing, eventually periodic, or $\erc{u}{v}$ is defined.\\
            Then the same holds for all $u,v \in S^*$.
        \end{lemma}
        \begin{proof}
            By induction on $l=\wl{u}_G + \wl{v}_G$. If $l=0,1$, then $u$ or $v$ is empty and $\erc{u}{v}$ is defined.
            Same if $l=2$, and $u$ or $v$ is empty. If $\wl{u}_G=\wl{v}_G=1$, then $u$ and $v$ are in $G$, and this is the assumption.\\
            Assume $l \ge 3$. Without loss of generality, we can assume that $\wl{v}_G \ge 2$. Then, we can write $v = v_1v_2$ with $v_1 \in G$ and $\wl{v_2}_G = \wl{v}_G - 1$, and the induction hypothesis applies to $(u,v_1)$. Note that we have $(u,v) \rf (u,v_1)$. Then, by Lemma \ref{lemm:trichotomy_rf}, if $(u,v_1)$ is failing (resp. eventually periodic), then $(u,v)$ is failing (resp. eventually periodic), and the lemma holds. If $(u,v_1)$ is neither failing nor eventually periodic, then $\erc{u}{v_1}$ is defined. Then there exist some $u',v_1'$ such that $\fin{u}v_1 \rr^* v_1'\fin{u'}$. Note now that we have $(u,v) \rf (u',v_2)$. Moreover, the Garside family properties imply that $\wl{u'}_G \le \wl{u}_G$. Thus, the induction hypothesis applies to $(u',v_2)$, and as before, if $(u',v_2)$ is failing (resp. eventually periodic), then $(u,v)$ is failing (resp. eventually periodic). If $(u',v_2)$ is neither failing nor eventually periodic, then $\erc{u'}{v_2}$ is defined, so  some $u'',v_2'$ such that $\fin{u'}v_2 \rr^* v_2'\fin{u''}$. Then, we have $\fin{u}v \rr^* v_1'v_2'\fin{u''}$. So we have $\erc{u}{v}$ is defined, and the lemma holds.
        \end{proof}
        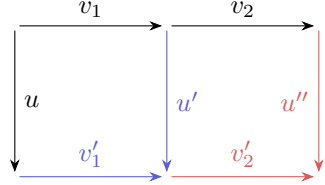
\begin{figure}[h!]
            \begin{center}
                \begin{tikzpicture}\label{figimported:trichotomy_ind}
                    \node[scale = 0.5] (P0) at (1, -1) {};
                    \node[scale = 0.5] (P1) at (3, -1) {};
                    \node[scale = 0.5] (P2) at (5, -1) {};
                    \node[scale = 0.5] (P3) at (1, -3) {};
                    \node[scale = 0.5] (P4) at (3, -3) {};
                    \node[scale = 0.5] (P5) at (5, -3) {};
                    \draw[-Stealth] (P0) edge [] node[auto]{${v_1}$} (P1);
                    \draw[-Stealth] (P0) edge [] node[auto]{$u$} (P3);
                    \draw[-Stealth] (P1) edge [] node[auto]{${v_2}$} (P2);
                    \draw[-Stealth] (P1) edge [color={rgb,255:red,92;green,92;blue,214}] node[auto]{${u'}$} (P4);
                    \draw[-Stealth] (P2) edge [swap, color={rgb,255:red,214;green,92;blue,92}] node[auto]{${u''}$} (P5);
                    \draw[-Stealth] (P3) edge [color={rgb,255:red,92;green,92;blue,214}] node[auto]{${v_1'}$} (P4);
                    \draw[-Stealth] (P4) edge [color={rgb,255:red,214;green,92;blue,92}] node[auto]{${v_2'}$} (P5);
                \end{tikzpicture}
            \end{center}
            \caption{Illustration of the proof of Lemma \ref{lemm:trichotomy_ind}.}
            \label{fig:trichotomy_ind}
        \end{figure} 
        Now the proof of Theorem \ref{theo:trichotomy} reduces to the case where $u$ and $v$ are in $G$.\\
        \begin{lemma}\label{lemm:trichotomy_gar}
            Let $\pres{S}{R}^+$ be a right-complemented presentation of a monoid $M$ for which right-reversing is complete. Assume $M$ has a finite Garside family $G$. Then for all $u,v \in G$, $(u,v)$ is either failing, eventually periodic, or $\erc{u}{v}$ is defined.
        \end{lemma}
        \begin{proof}
            The key is that an infinite sequence of reductions in the finite set $G\times G$ are necessarily eventually periodic. We denote by $F$ the set of failing pairs of $G\times G$, by $T$ the set of pairs $(u,v)$ such that $\erc{u}{v}$ is defined, and by $E$ the rest of the pairs, i.e. $E = G\times G \setminus (F \cup T)$. Then, it is enough to prove that for every pair $(u,v) \in E$, there exists a pair $(u',v') \in E$ such that $(u,v) \strrf (u',v')$.\\
            Let $(u,v) \in E$. Then $u$ and $v$ are necessarily not empty (otherwise, $(u,v) \in T )$. Then we can write $u = u_1u_2$ and $v = v_1v_2$ with $u_1,v_1 \in S$ and $u_2,v_2 \in G$ because $G$ is closed under right-divisor. Moreover, $\rc{u_1}{v_1}$ is defined, or otherwise $(u,v) \in F$. Then, $\fin{u}v \rr^1 \fin{u_1}v_1'{v_1}\fin{u_1'}v_2$. Then, we have $(u,v) \strrf (u_2,v_1')$ and $(u,v) \strrf (u_1',v_2)$. Now, both pairs $(u_2,v_1')$ and $(u_1',v_2)$ can't be in $F$ because of Lemma \ref{lemm:trichotomy_rf}. If one of them is in $E$, we are done. Otherwise, assume that both are in $T$. Then, there exist some $u_2',v_1'' \in G$ such that $\fin{u_2}v_1' \rr^* v_1''\fin{u_2'}$. And there exist some $u_1'',v_2' \in G$ such that $\fin{u_1'}v_2 \rr^* v_2'\fin{u_1''}$. So we have $\fin{u}v \rr^* v_1''\fin{u_2'}v_2'\fin{u_1''}$. So we have $(u,v) \strrf (u_2',v_2')$. Now, $(u_2',v_2')$ cannot be in $F$ for the same reason, and if it is in $T$, then there exist some $u_2'',v_2''$ such that $\fin{u_2'}v_2' \rr^* v_2''\fin{u_2''}$. Then, we have $\fin{u}v \rr^* v_1''v_2''\fin{u_2''}\fin{u_1''}$, so $(u,v) \in T$ which contradicts the assumption. So we have $(u_2',v_2') \in E$, and this concludes the proof of Proposition \ref{theo:trichotomy}.
        \end{proof}
        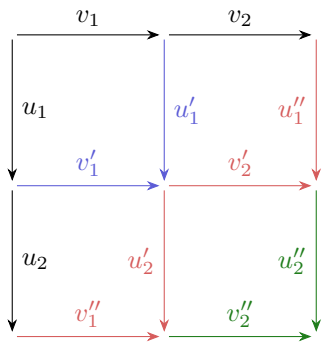
\begin{figure}[h!]
            \begin{center}
                \begin{tikzpicture}\label{figimported:trichotomy_gar}
                    \node[scale = 0.5] (P0) at (1, -1) {};
                    \node[scale = 0.5] (P1) at (3, -1) {};
                    \node[scale = 0.5] (P2) at (5, -1) {};
                    \node[scale = 0.5] (P3) at (1, -3) {};
                    \node[scale = 0.5] (P4) at (3, -3) {};
                    \node[scale = 0.5] (P5) at (5, -3) {};
                    \node[scale = 0.5] (P6) at (1, -5) {};
                    \node[scale = 0.5] (P7) at (3, -5) {};
                    \node[scale = 0.5] (P8) at (5, -5) {};
                    \draw[-Stealth] (P0) edge [] node[auto]{${v_1}$} (P1);
                    \draw[-Stealth] (P0) edge [] node[auto]{${u_1}$} (P3);
                    \draw[-Stealth] (P1) edge [] node[auto]{${v_2}$} (P2);
                    \draw[-Stealth] (P1) edge [color={rgb,255:red,92;green,92;blue,214}] node[auto]{${u_1'}$} (P4);
                    \draw[-Stealth] (P2) edge [swap, color={rgb,255:red,214;green,92;blue,92}] node[auto]{${u_1''}$} (P5);
                    \draw[-Stealth] (P3) edge [color={rgb,255:red,92;green,92;blue,214}] node[auto]{${v_1'}$} (P4);
                    \draw[-Stealth] (P3) edge [] node[auto]{${u_2}$} (P6);
                    \draw[-Stealth] (P4) edge [color={rgb,255:red,214;green,92;blue,92}] node[auto]{${v_2'}$} (P5);
                    \draw[-Stealth] (P4) edge [swap, color={rgb,255:red,214;green,92;blue,92}] node[auto]{${u_2'}$} (P7);
                    \draw[-Stealth] (P5) edge [swap, color={rgb,255:red,32;green,125;blue,28}] node[auto]{${u_2''}$} (P8);
                    \draw[-Stealth] (P6) edge [color={rgb,255:red,214;green,92;blue,92}] node[auto]{${v_1''}$} (P7);
                    \draw[-Stealth] (P7) edge [color={rgb,255:red,32;green,125;blue,28}] node[auto]{${v_2''}$} (P8);
                \end{tikzpicture}
            \end{center}
            \caption{Illustration of the proof of Lemma \ref{lemm:trichotomy_gar}.}
        \end{figure}
    \subsection{Consequences}
        The previous results along with the fact that eventual periodicity and failing can be detected algorithmically, gives the following result.
        \begin{proposition}\label{prop:comput_lcm}
            Let $\pres{S}{R}^+$ be a right-complemented presentation of a monoid $M$ for which right-reversing is complete, $M$ is noetherian and has a finite Garside family. Then there exists an algorithm that computes (a word on $S$ that represents) the right-lcm of two words $u,v \in S^*$, or detects if they have no common right-multiple.
        \end{proposition}
        Moreover, knowing that a finite Garside family exists, the previous algorithm can be used to compute a minimal Garside family, based on algorithm 2.23 in \cite{dehornoy_algorithms_2014}.
        \begin{proposition}\label{prop:compute_garside}
            Let $\pres{S}{R}^+$ be a right-complemented presentation of a monoid $M$ for which right-reversing is complete, $M$ is noetherian. Starting with $G_0 = S$, we can iterate the process of computing the right-lcm : $G_{n+1} = G_n \cup \{\lcm(g,h) | g,h \in G_n\}$. We stop if $G_{n+1} = G_n$ for some $n$. Then the process ends if and only if $M$ has a finite Garside family, and in this case, $G_n$ is a minimal Garside family of $M$.
        \end{proposition}
        Computing the Garside family with this algorithm can be very long, and our implementation (\url{https://github.com/domyrmininon/compute_gf}) that improves a little bit the efficiency of the right-reversing by remembering previous computations (at the cost of memory consumption) allows to compute Garside families of size up to around 2000 in less than an hour. We recovered the results of \cite{dehornoy_garside_2015} for the size of the smallest Garside family, denoted $F$, for several affine type Artin-Tits groups. We checked for small cases that our algorithm gives $|F| = |W|$ for spherical types. We complete Table 1 in \cite{dehornoy_garside_2015} with type $\widetilde{G_2}$, and $\widetilde{D_4}$.
            
        \begin{table}[htb]\centering
        \begin{tabular}{|c|c|c|c|c|c|c|c|c|c|}
        \VR(4,2) type of $(W, S)$\null
        & spherical \null
        &$\widetilde{\mathrm{A}}_2\null$
        &$\widetilde{\mathrm{A}}_3\null$
        &$\widetilde{\mathrm{A}}_4\null$
        &$\widetilde{\mathrm{B}}_3\null$
        &$\widetilde{\mathrm{C}}_2\null$
        &$\widetilde{\mathrm{C}}_3\null$
        &$\widetilde{\mathrm{G}}_2\null$
        &$\widetilde{\mathrm{D}}_4\null$\\
        \hline
        \VR(4,2)$|E|$&$1$&$3$&$10$&$35$&$14$&$3$&$12$&$3$&$49$\\
        \hline
        \VR(4,2)$|F|$&$|W|$&$16$&$125$&$1{,}296$&$315$&$24$&$317$&$41$&$2400$
        \end{tabular}
        \vspace{2mm}
        \caption[]{Sizes of the smallest Garside family $F$ in the Artin-Tits monoid associated with the Coxeter system $(W, S)$. By definition, when $F$ is finite, it must be the set of all right-divisors of a finite set of elements (the \enquote{extremal elements}) that we denote by $E$.}
        \label{T:Size}
        \end{table}
        In spherical type, the extremal element is the Garside element $\Delta$ which can easily be obtained by closing the set of generators under only lcm. In general, to compute the set of extremal elements, one needs to take right-complement at some point. Since the Garside family is fully described by this set of extremal elements, it might be an intersting object to investigate. Is there a way to compute it directly ? Have these extremal element particular agebraic property ? We conclude this paper with illustration of the 3 smallest Garside families for the affine type Artin-Tits monoids of rank 3: $\widetilde{\mathrm{A}}_2\null$, $\widetilde{\mathrm{C}}_2\null$, $\widetilde{\mathrm{G}}_2\null$.

        \begin{figure}[ht]
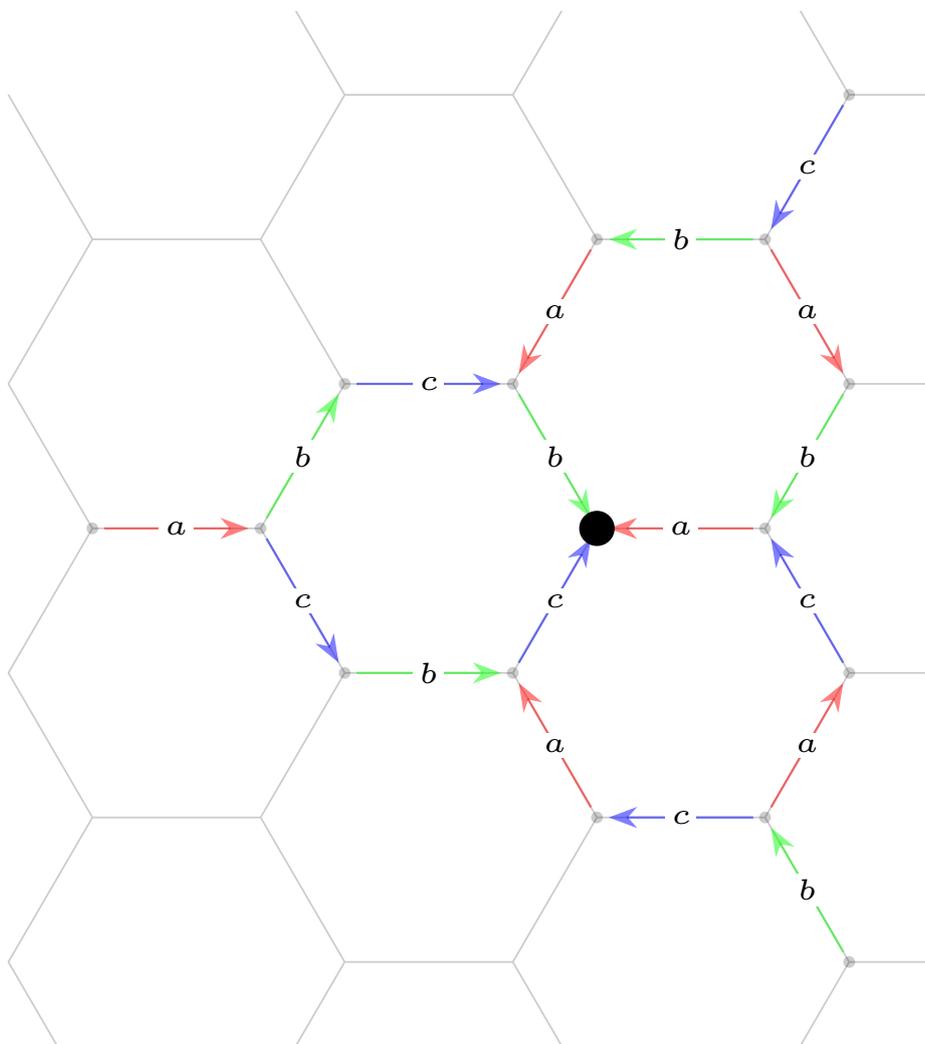

            \resizebox{\linewidth}{!}{

            }
            \caption{The smallest Garside family for the affine type Artin-Tits monoid of type $\widetilde{\mathrm{A}}_2$. Elements of the family are the paths ending at the origin. The diagram fits on the hexagonal tiling of the plane, which is the Cayley graph of the corresponding Coxeter group. Starting from the extremities, one can read the $3$ extremal elements : $abcb$, $bcac$ and $caba$.} 
        \end{figure}

        \begin{figure}[ht]
            \resizebox{\linewidth}{!}{
                %
            }
            \caption{The smallest Garside family for the affine type Artin-Tits monoid of type $\widetilde{\mathrm{C}}_2$. The diagram fits on a truncated square tiling of the plane, which is the Cayley graph of the corresponding Coxeter group. The $3$ extremal elements are : $bac$, $cbabcbc$ and $abcbaba$.} 
        \end{figure}

        \begin{figure}[ht]
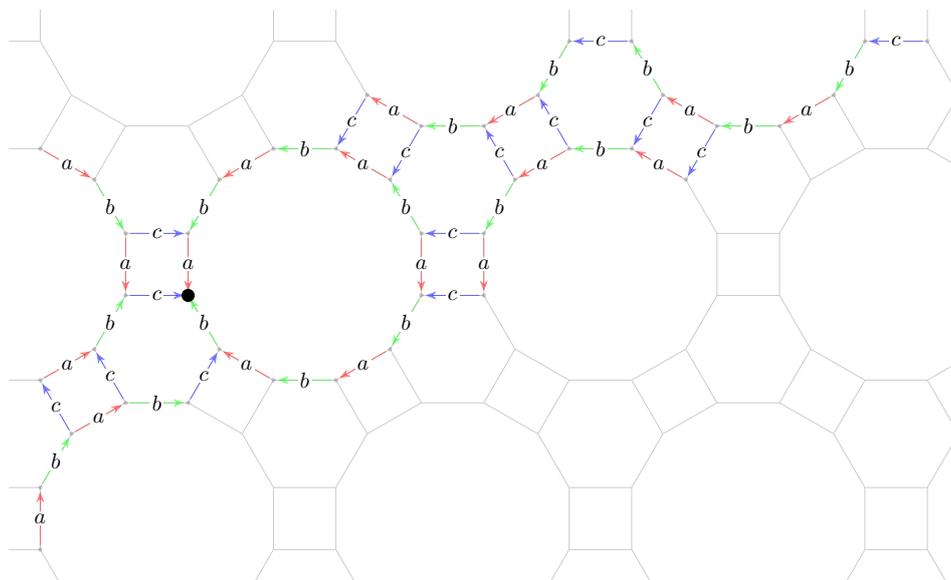
      
            \resizebox{\linewidth}{!}{
                %
            }
            \caption{The smallest Garside family for the affine type Artin-Tits monoid of type $\widetilde{\mathrm{G}}_2$. Here, the diagram fits on a semi-regular tiling of the plane, called the \emph{truncated trihexagonal tiling} which is the Cayley graph of the corresponding Coxeter group. There is again $3$ extremal elements : $abca$, $abcabc$, $cbabcababacbabab$} 
        \end{figure}

\clearpage

\bibliographystyle{ws-ijac}
\bibliography{bibli} 

\end{document}